\documentclass[reqno,12pt]{amsart}
\usepackage{amsfonts,amssymb,latexsym,amsmath,amsthm,a4}
\usepackage{graphicx}
\usepackage{pst-all}

\newtheorem{Theorem}{Theorem}[section]
\newtheorem{Lemma}[Theorem]{Lemma}

\newtheorem{Corollary}[Theorem]{Corollary}
\newtheorem{Proposition}[Theorem]{Proposition}
\theoremstyle{definition}
\newtheorem{Remark}[Theorem]{Remark}
\newtheorem{Definition}[Theorem]{Definition}
\newtheorem{Example}[Theorem]{Example}

\numberwithin{equation}{section}
\numberwithin{figure}{section}

\begin{document}

\title[Antisymmetry of solutions for some weighted elliptic problems]
{Antisymmetry of solutions for some weighted elliptic problems}

\author{Xavier Cabr\'e}
\address{X. Cabr\'e \textsuperscript{1,2},
\textsuperscript{1} Universitat Polit\`ecnica de Catalunya, Departament de Mate-\-m\`{a}tiques, Diagonal 647, 08028 Barcelona, Spain.
\textsuperscript{2} ICREA, Pg. Lluis Companys 23, 08010 Barcelona, Spain}
\email{xavier.cabre@upc.edu}

\author{Marcello Lucia}
\address{M. Lucia, The City University of New York, CSI,  Mathematics Department, 2800 Victory Boulevard, Staten Island,
New York 10314, USA}
\email{mlucia@math.csi.cuny.edu}

\author{Manel Sanch\'on}
\address{M. Sanch\' on,
Grupo AIA - Aplicaciones en Inform\'atica Avanzada, SL, ESADECREAPOLIS Planta 2a Bloc C Portal 1,
Av. Torre Blanca 57, CP-08172 Sant Cugat del Vall\`es, Spain}
\email{manel.sanchon@gmail.com}

\author{Salvador Villegas}
\address{S. Villegas, Universidad de Granada, Departamento de An\'alisis Matem\'atico, Facultad de Ciencias, Campus Fuentenueva, 18071 Granada, Spain}
\email{svillega@ugr.es}

\thanks{The first and third authors were supported by MINECO grant MTM2014-52402-C3-1-P.
They are part of the Catalan research group 2014 SGR 1083. The first author is
member of the Barcelona Graduate School of Mathematics.
The second author was supported by the Simons Foundation Collaboration
Grant for Mathematicians 210368. The third author was also supported by ERC grant 320501
(ANGEOM project).
The fourth author was supported by grant MTM2015-68210-P}

%
%
%
\begin{abstract}
This article concerns the antisymmetry, uniqueness, and
monotonicity properties of solutions to some elliptic functionals involving weights
and a double well potential. In the one-dimensional case, we introduce the
continuous odd rearrangement of an increasing function and we show that it
decreases the energy functional when the weights satisfy a certain convexity-type
hypothesis. This leads to the antisymmetry or oddness of increasing solutions
(and not only of minimizers).
We also prove a uniqueness result (which leads to antisymmetry) where a
convexity-type condition by Berestycki and Nirenberg on the weights is improved to a
monotonicity condition. In addition, we provide with a large class of problems
where antisymmetry does not hold. Finally, some rather partial extensions
in higher dimensions are also given.
\\

\noindent
{\sc Mathematics Subject Classification 2010:}
35J61,    
35B06,   
35B07,   
35Q92   
\\
{\sc Key words:}
Bistable nonlinearity, weights, antisymmetric solutions, continuous odd rearrangement,
monotonicity, uniqueness.
\end{abstract}

\maketitle

\section{Introduction}\label{section1}


Symmetry properties of solutions to nonlinear elliptic problems have been
extensively studied in the literature. For Dirichlet problems with zero
boundary conditions, the Steiner and Schwarz symmetrizations (see~\cite{Kawohl,PolyaSzego})
and the moving planes method~\cite{Alexandrov,GNN} have been successfully applied
to derive symmetry, with respect to a hyperplane, of minimizers or of positive
solutions to many nonlinear problems. For sign changing solutions, for instance
still with zero Dirichlet boundary conditions, it is well known that the symmetry
with respect to a hyperplane may fail. For this, simply consider the Dirichlet
eigenfunctions of the Laplacian in an interval or a ball. Instead, for some of
them, what holds is antisymmetry, as defined next.

A natural question that we address here is whether solutions are {\it antisymmetric}
or {\it odd} with respect to a  hyperplane (and also with respect to certain cones,
as we will see later) whenever the problem is invariant under
the odd reflection of the solution. Aside being interesting for its own sake
(and for a possible answer to an open problem presented below), symmetry
and antisymmetry of solutions of PDEs are important in physics and other fields of
mathematics. For instance, in quantum mechanics, a system of identical bosons
(respectively, fermions) is described by a multiparticle wavefunction which is
symmetric (respectively, antisymmetric) under the interchange of pairs of particles.

For nonzero Dirichlet boundary data, Berestycki and Nirenberg~\cite{BeresNiren} used
the maximum principle, together with different versions of their sliding method, to
give some sufficient conditions that guarantee solutions to be unique and antisymmetric
with respect to a hyperplane passing through the origin.
In \cite{WeiWinter:2005}, Wei and Winter showed that two-peaks nodal solutions  to
$\varepsilon \Delta u - u + |u|^{p-2} u = 0$ in a ball, with zero Dirichlet boundary
conditions, are antisymmetric (with respect to a hyperplane through the origin) when $\varepsilon$ is small.
Extremals of the ratio $\|\nabla u\|_2/ \|u \|_p$ for functions of average zero in a
ball (Neumann boundary conditions) have been considered by Gir\~ao and Weth in~\cite{GiraoWeth},
who showed that they are antisymmetric (with respect to a well chosen hyperplane through the
origin) for $p$ close to $2$, while this is not anymore the case for large $p$.
In~\cite{GT}, Grumiau and Troestler proved that, for $p$ close to $2$,
the least energy nodal solution of $\Delta u + |u|^{p-2} u = 0$ with
zero Dirichlet boundary condition in a ball or an annulus is
unique (up to rotation and multiplicative constant $\pm 1$) and antisymmetric
with respect to a hyperplane passing through the origin.

Our main motivation to study the antisymmetry of solutions is driven by one conjecture
posed by De Giorgi~\cite{DeGiorgi} in 1978. The following is one of its  natural formulations:
{\it Let $u \in C^2(\mathbb R^N)$ be a bounded function which is, on each bounded domain
$\Omega \subset \mathbb R^N$, a minimizer $($under perturbations with compact support in
$\Omega$$)$ of the Allen-Cahn functional
\begin{equation}\label{allencahn}
    E (u, {\Omega}) =  \int_{\Omega} \left\{ \frac{1}{2}|\nabla u|^2 + G(u) \right\} dx  ,
\end{equation}
where $G (u) = (1 - u^2)^2/4$. Is it true that the level sets of $u$ are hyperplanes,
at least if $N\leq 7$}?

Throughout the paper, by minimizer we always mean ``absolute minimizer''.

After the first results in dimensions 2 and 3 in \cite{GhoussoubGui,AmbrosioCabre,AAC},
a breakthrough came with the work by Savin~\cite{Savin} who showed that the above
conjecture is indeed true up to dimension $N\leq 7$. Later, for $N = 9$ del Pino,
Kowalczyk, and Wei~\cite{DelPinoKowalczykWei} constructed a solution~$u$
that is monotone in the direction $x_9$, has limit $\pm 1$
as $x_9 \to \pm \infty$, and has level sets which are not hyperplanes.
A result from~\cite{AAC} guarantees that such
monotone solution is in fact a minimizer of the functional $E$,
providing a counter-example to the above conjecture in dimensions
$N \geq 9$. More recently, Liu, Wang, and Wei~\cite{LWW} have shown the existence
of a minimizer when $N=8$ with level sets that are not hyperplanes. In
dimension 8, however, an important open question that we describe next remains open.

The conjecture of De Giorgi was motivated by a classical result on minimal surfaces. While
every minimizing minimal surface in all of $\mathbb{R}^N$ must be a hyperplane
if $N\leq 7$, Bombieri, De Giorgi, and Giusti \cite{BdGG} established that
\textit{the Simons cone}
$$
\mathcal{C}:=\{(x^1,x^2)\in\mathbb{R}^4\times\mathbb{R}^4:|x^1|=|x^2|\}
$$
is a minimizing minimal surface in $\mathbb{R}^8$ different from a hyperplane.
Therefore, in dimension 8, the canonical counter-example to the conjecture
of De Giorgi should be given by the so-called \textit{saddle-shaped solution} $u$
to the Euler-Lagrange equation of \eqref{allencahn}, \textit{i.e.}, $-\Delta u=u-u^3$.
Namely, a solution $u=u(x^1,x^2)$, where  $(x^1,x^2) \in
\mathbb R^{m} \times \mathbb R^{m}$, in even dimension
$N=2m$ which is radially symmetric in the first $m$ variables and also in the last $m$ variables (\textit{i.e.},
$u=u(|x^1|,|x^2|)$) and antisymmetric under the reflection $\sigma (x^1,x^2) = (x^2,x^1)$ (\textit{i.e.},
$u(|x^2|,|x^1|)=-u(|x^1|,|x^2|)$). In particular, its zero level set is the
Simons cone $\mathcal{C}$ above and
$u$ is odd with respect to $\mathcal{C}$.
While the existence of such antisymmetric solution in dimension $N=2m$ is easy to establish
(\cite{CT1,CT2}), its uniqueness is a more delicate issue and has been established more recently
by the first author \cite{Cabre}. The remaining open problem is the following:

\vspace{2mm}
\textbf{Open question 1.} {\it Is the saddle-shaped
solution a minimizer of $- \Delta u = u - u^3$ in dimensions $N=2m \geq 8$}?
\vspace{2mm}

The saddle-shaped solution in $\mathbb{R}^{2m}=\mathbb{R}^m\times\mathbb{R}^m$
is a function of the two radial variables $s=|x^1|$ and $t=|x^2|$, $u=u(s,t)$.
In these variables the energy functional (up to a multiplicative constant) reads
\begin{equation}\label{omegaR}
  E (u,\Omega_R)
  = \int_{\Omega_R} \Big\{ \frac{1}{2} |\nabla u|^2 + G(u)  \Big\} s^{m-1} t^{m-1} ds\,dt.
\end{equation}
This functional is invariant under odd reflection in the diagonal $\{s=t\}$,
which is the Simons cone $\mathcal{C}$.
Here, for instance we may take $\Omega_R:=\{s>0,\ t>0,\ s^2+t^2<R^2\}$ to be a quarter
of ball in the plane. The saddle-shaped solution is antisymmetric or odd with
respect to $\{s=t\}$. The following open problem will be connected with Open Question~1.

\vspace{2mm}
\textbf{Open question 2.} {\it Are minimizers of \eqref{omegaR} $($for all, or at least for some,
Dirichlet boundary conditions on $\{s>0,\ t>0,\ s^2+t^2=R^2\}$ which are antisymmetric with
respect to $\{s=t\})$ also antisymmetric when $2m\geq 8$ and $R$ is large enough}?
\vspace{2mm}

A positive answer to Open question 2 leads to the corresponding positive answer to Open question 1.
Indeed, if antisymmetry of minimizers holds for the problem in $\Omega_R$ then, by letting $R\rightarrow\infty$,
one obtains an antisymmetric solution in all of $\mathbb{R}^{2m}$ which is a minimizer
(being limit of minimizers in $\mathbb{R}^{2m}$). In particular, this solution
being a minimizer, one easily shows that it is not identically zero (see~\cite{CT1,CT2}).
Thus, by the uniqueness result of \cite{Cabre}, it is the saddle-shaped solution.

Therefore, Open question 2 has a negative answer in dimensions 2, 4, and 6,
since in these dimensions the saddle-shaped solution is known not to be a minimizer
(for instance by the results of Cabr\'e and Terra~\cite{CT1,CT2} on instability of the
saddle solution, or by Savin's \cite{Savin} result).

Note the presence of the weight $s^{m-1} t^{m-1}$ in the energy functional above.
Alternatively, considering coordinates $y=(s+t)/\sqrt{2}$ and $z=(s-t)/\sqrt{2}$,
we would be concerned with oddness in the variable $z$ in the presence of
the weighted measure
$$
2^{m-1}s^{m-1}t^{m-1}\,dsdt=(y^2-z^2)^{m-1}\,dy\,dz,
$$
which is even in $z$, where $z\in [-y,y]$. Note that this weight (as a function of $z$)
is not increasing in $[0,y]$ ---while being increasing is the condition that leads to oddness
(at least in dimension 1) in one of our results, Theorem~\ref{thm:Brahms2}.

Other questions regarding the weighted measure $s^{m-1}t^{m-1}\,ds\,dt$
(or, more generally, $x_1^{A_1}\cdots x_n^{A_N}\,dx_1\cdots dx_n$ coming from multiple
radial symmetries) have been recently studied in \cite{CR-01,CR-02}. They concern sharp
weighted isoperimetric and Sobolev inequalities and were originated from the study of
extremal solutions in explosion (or Gelfand type) problems.

With this motivation in mind, we are led to understand the antisymmetry
of critical points of functionals involving weights. Our paper presents
alternative ways of proving antisymmetry of minimizers and provides
several new uniqueness results for variational problems with weights.
Our main results apply to one-dimensional problems. Some partial
answers in the higher dimensional case ---which however do not allow to
solve the motivating open questions above--- are presented later in this section.

\subsection{One-dimensional case}
In the one-dimensional case, given functions $a$ and $b$ defined on an interval
$[-L, L]$ and a function $G$ satisfying
\begin{equation} \label{eq:A}
\left.
\begin{array}{c}
a,b:[-L,L] \to \mathbb{R} \textrm{ are positive and even }C^1([-L,L])
\textrm{ functions},
   \vspace{3mm}\\
G: \mathbb R \to \mathbb R \textrm{ is a nonnegative and even }C^1([-L,L]) \textrm{ function},
\end{array}
\right\}
\end{equation}
we consider the energy functional
\begin{equation}\label{functional}
  \mathcal{E} (u, (-L,L)) := \int_{-L}^L \left\{\frac{1}{2}(u')^2a(x)+G(u)b(x)\right\} dx
\end{equation}
in
$$
  H^1_m ((-L,L)) := \left\{ u \in H^1 ((-L,L)) \, \colon \,
  u(-L) = -m,\ u(L)=m \right\},
$$
where $m\geq 0$ is given.

Critical points of this functional are solutions of the associated
Euler-Lagrange equation
\begin{equation}\label{E-L}
\left\{
\begin{array}{l}
  -\left(a u'\right)' = b f(u)   \qquad  \hbox{ in } (-L,L),
  \\
  u(L)=-u(-L)= m,
\end{array}
\right.
\end{equation}
where
$$
f=-G'
$$
is an odd nonlinearity. Note that $G$ is defined up to an additive constant and,
therefore, the hypothesis ``$G$ is nonnegative'' in \eqref{eq:A} can be replaced by
``$G$ is bounded from below''.

We define the \textit{flipped} $u_\star$ of a continuous function $u$ in $[-L,L]$ as
\begin{equation}\label{def:flipped}
u_\star(x):=-u(-x) \quad \textrm{for }x\in[-L,L].
\end{equation}
Note that if $u$ is a solution of \eqref{E-L}, its flipped
is also a solution under assumption \eqref{eq:A} (see~Figure~\ref{fig1}).
In addition, $u$ is antisymmetric or odd if and
only if $u=u_\star$. Note also that $\mathcal{E}(u,(-L,L))=\mathcal{E}(u_\star,(-L,L))$
under assumption \eqref{eq:A}.

After an appropriate change of variables $y=\gamma(x)$ (see
\eqref{eq:ChangeVariable} in Section~\ref{subsec2:2}), one
can always reduce the problem either to the case $a\equiv b$ or to
the case $b\equiv 1$ ---something that sometimes will be useful. When
$a\equiv b$, the equation in \eqref{E-L} reads
\begin{equation}\label{nondiv}
-u''-({\rm log}\ a)'u'=f(u)\quad \textrm{in }(-L,L).
\end{equation}
For this last equation, Berestycki and Nirenberg~\cite{BeresNiren} used
several versions of their sliding method to prove uniqueness and antisymmetry results. In the
one-dimen-\-sional case, one of their results states the following. It requires the
first order coefficient $({\rm log}\ a)'$ in \eqref{nondiv} to be nondecreasing.
%
\begin{Theorem} [Berestycki-Nirenberg~\cite{BeresNiren}, Theorem~4.1 and Corollary~4.3]
\label{thm:BerestyckiNirenberg}
Let us assume that \eqref{eq:A} holds, $a \equiv b$, and that $f$ is locally Lipschitz.
Let $L$ and $m$ be positive numbers. If
$$
{\rm log}\,a \textrm{ is convex}
$$
then there exists at most one solution to~\eqref{E-L} satisfying
\begin{equation}\label{apriori}
    -m \leq u \leq m \quad\textrm{in }[-L,L],
\end{equation}
and this solution, if it exists, is odd and increasing.
\end{Theorem}
In higher dimensions, an analogous result was proved also in the same paper \cite{BeresNiren}. In fact,
when the domain $\Omega$ is a cylinder $(-L,L)\times\omega$, with $\omega\subset\mathbb{R}^{N-1}$,
they proved monotonicity in the $x_1$ variable, as well as uniqueness and antisymmetry
of solutions of $-\Delta u=h(x,u,\nabla u)$ under suitable symmetry and monotonicity
assumptions on the boundary data and on $h(x,q,p)$.
The main ingredient in their proof of Theorem \ref{thm:BerestyckiNirenberg} is
a parabolic version of the sliding method. They compare translations of the solution
with the solution itself and then apply the maximum principle to obtain monotonicity and
uniqueness of solutions (see the proof of our Proposition~\ref{prop:BerNir} for this kind of argument).
In~\cite{BeresNiren}, it is also observed that, by the maximum principle,
the \textit{a priori} bound \eqref{apriori} in the above theorem is
automatically satisfied by every solution $u$ of \eqref{E-L} (with $a\equiv b$) if for instance
one assumes
%
$$
    G'\geq0 \quad \textrm{in } (m,\infty)=(u(L),\infty).
$$
This is the same as assuming $f\leq0$ in $(m,\infty)$ ---recall that here we assume
$G$ to be even and hence the nonnegativeness of $f$ also in $(-\infty,-m)$ follows.

Our results will complete in several ways, in the one-dimensional case, the above statement
of Berestycki and Nirenberg.
Theorem~\ref{thm:BerestyckiNirenberg} assumes log-convexity of the weight $a$
but only \eqref{eq:A} for the potential $G$ (\textit{i.e.}, that $G$ is even). If, instead, one assumes only
\eqref{eq:A} on the weight $a$  (\textit{i.e.}, that $a$ is even) but also that $G$ is convex, then we also have
uniqueness of solution. This is clear since the energy functional $\mathcal{E}$
will be convex in this case. Our first result improves Theorem~\ref{thm:BerestyckiNirenberg}
by replacing the assumption on log-convexity of $a$ by only the monotonicity of $a$,
at the price of assuming also monotonicity of $G$ in $(0,m)$. In addition, we do not require
the \textit{a priori} assumption \eqref{apriori} on the solution. More precisely,
we establish the following.

\begin{Theorem}  \label{thm:Brahms2}
Let $L$ and $m$ be positive numbers.
Assume that \eqref{eq:A} holds, $a \equiv b$, and that $f=-G'$ is locally Lipschitz.
Suppose further that
\begin{equation}\label{hyp:thm:Brahms2}
   a' \geq 0 \hbox{ in } (0,L),
   \quad
   G \geq G(m) \hbox{ in }  (0, \infty),
   \quad \textrm{and}\quad
   G' \leq 0 \hbox{ in } (0, m).
\end{equation}
Then, problem~\eqref{E-L} admits a unique solution, which is therefore odd.
Furthermore, this solution is increasing.
\end{Theorem}
Note that $G'\leq 0$ in $(0,m)$ is simply the hypothesis $f\geq 0$ in $(0,m)$ on the
nonlinearity. It holds for instance in our model case $f(u)=u-u^3$, $G(u)=(1-u^2)^2/4$,
and $m=1$. The other hypothesis, $G\geq G(m)$ in $(0,\infty)$, is also satisfied in
this case.

We will prove Theorem~\ref{thm:Brahms2} for general weights $a$ and $b$ (not
necessarily equal). In this general case the first assumption in \eqref{hyp:thm:Brahms2}
becomes
\begin{equation}\label{ab:inc}
(ab)'\geq 0 \quad\textrm{in }(0,L)
\end{equation}
(in Section \ref{subsec2:2} we explain how one can reduce the problem either
to the case $a\equiv b$ or to $b\equiv 1$).
Our proof uses the Hamiltonian function
\begin{equation}\label{L:intro}
\mathcal{H}(x,q,p):=\frac{1}{2}(a(x)p)^2-a(x)b(x)G(q), \quad
(x,q,p)\in(-L,L)\times \mathbb{R}^2.
\end{equation}
We use it first to prove that any solution $u$ of \eqref{E-L} is increasing if
both~\eqref{ab:inc} and the second assumption in \eqref{hyp:thm:Brahms2} hold. Then,
we are able to prove uniqueness, and hence antisymmetry, of solutions $u$
using the identity $\frac{d}{dx}\mathcal{H}(x,u,u')=-(ab)'G(u)$.

Our second main contribution consists in deriving antisymmetry of
solutions by using a new continuous odd rearrangement. Here we will need
the log convexity assumption, as in the Berestycki-Nirenberg result.
Making a change of variables (see Section \ref{subsec2:2}) we can assume
$b\equiv 1$. In this case, given an \textit{increasing}
function $v \in H^1_m ((-L,L))$, let us call $\rho=\rho(\lambda)$ its
inverse function, \textit{i.e.}, $v(\rho(\lambda))=\lambda$ for all $\lambda
\in[-m,m]$ (see Figure \ref{fig1}).
Recall that functions in $H^1((-L,L))$ are continuous, and thus, the hypothesis
of being increasing is justified. We define (see Definition~\ref{def:rear} below)
the \textit{continuous odd rearrangement} $\{v^t\}$
of $v$, with $0\leq t\leq 1$, as the inverse function $v^t(x)=\lambda$ of
$$
\begin{array}{lcl}
\rho^t(\lambda)&:=&t\rho(\lambda)+(1-t)(-\rho(-\lambda))
\\
&=&t\rho(\lambda)+(1-t)\rho_\star(\lambda)\qquad\textrm{ for all }
\lambda\in [-m,m].
\end{array}
$$
Note that $\rho_\star(\lambda)=-\rho(-\lambda)$, the flipped of $\rho$, is
the inverse function of the flipped $v_\star$ of $v$. For $t=1$ and $t=0$,
$v^t$ coincides respectively with $v$ and its flipped $v_\star$: $v^1=v$ and $v^0=v_\star$.
Moreover, for $t=1/2$, $v^{1/2}$ is always an odd function. We call it
the \textit{odd rearrangement} of $v$.

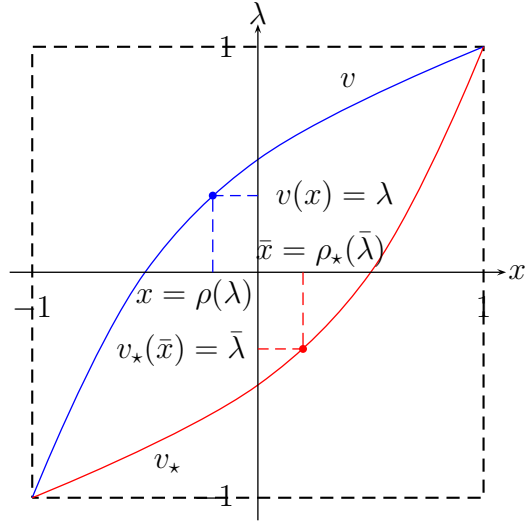
\begin{figure}[ht]\label{fig1}
\begin{center}
\psset{xunit=3cm,yunit=3cm}
\begin{pspicture}(0,1.2)

\psaxes[linestyle=dashed,axesstyle=frame,ticks=none]{->}(0,0)(-1,-1)(1,1)
\pscurve[linecolor=blue,linewidth=0.5pt](-1,-1)(-0.5,0)(0,0.5)(1,1)
\pscurve[linecolor=red,linewidth=0.5pt](-1,-1)(0,-0.5)(0.5,0)(1,1)
\psline[linewidth=0.5pt]{->}(-1.1,0)(1.1,0)
\psline[linewidth=0.5pt]{->}(0,-1.1)(0,1.1)
\rput(0.4,0.85){$v$}
\rput(-0.4,-0.85){$v_\star$}
\rput(1.15,0){$x$}
\rput(0,1.15){$\lambda$}

\psline[linecolor=blue,linestyle=dashed,linewidth=0.5pt]{-*}(-0.2,0)(-0.2,0.34)
\psline[linecolor=blue,linestyle=dashed,linewidth=0.5pt](-0.2,0.34)(0,0.34)
\rput(-0.28,-0.10){$x=\rho(\lambda)$}
\rput(0.34,0.34){$v(x)=\lambda$}

\psline[linecolor=red,linestyle=dashed,linewidth=0.5pt]{-*}(0.2,0)(0.2,-0.34)
\psline[linecolor=red,linestyle=dashed,linewidth=0.5pt](0,-0.34)(0.2,-0.34)
\rput(-0.34,-0.34){$v_\star(\bar x)=\bar \lambda$}
\rput(0.28,0.1){$\bar x=\rho_\star(\bar \lambda)$}
\end{pspicture}
\end{center}
\vspace{2.7cm}
  \caption{A function $v$ and its flipped $v_\star$.}
\end{figure}

One property of the continuous odd rearrangement is that $v$ and $v^t$ are
equidistributed with respect to the weight $b\equiv1$,
\textit{i.e.},
$$
  \left| \{ - \lambda < v^t  < \lambda  \} \right|
  =
  \left| \{ - \lambda < v < \lambda  \} \right|
	=
	\rho(\lambda)-\rho(-\lambda)
$$
for all $t\in[0,1]$ and $\lambda\in[0,m]$.
In particular, the integral $\int_{-L}^L G(v)\ dx$ is preserved under this rearrangement
for every \textit{even} continuous function $G$.

For a general positive weight $b$ we define the continuous odd rearrangement $\{v^t\}$
of $v$ with respect to $b$, with $0\leq t\leq 1$, as the inverse function $v^t(x)=\lambda$ of
$$
\rho^t(\lambda):=B^{-1}\Big(t B(\rho(\lambda))+(1-t)B(\rho_\star(\lambda))\Big)
\quad\textrm{ for all }\lambda\in [-m,m],
$$
where $B(x):=\int_0^x b(y)\, dy$. In this case, $v$ and $v^t$ are also
equidistributed with respect to the weight $b$, and therefore, the integral
$\int_{-L}^L G(v) b(x)\, dx$ is preserved under this rearrangement when $G$ is
continuous and even.

Our main result states that continuous odd rearrangement decreases the
kinetic energy under the hypothesis that $a\equiv b$ is log-convex.
We assume that the given function $v$ is increasing, as in the previous definition.
\begin{Theorem} \label{thm:OddRearrangement}
For $L>0$, $m>0$, let $a$, $b \in C^0([-L,L])$ and $G\in C^0(\mathbb{R})$ be even
functions and
$v \in C^1_m ([-L,L]):=\left\{ w \in C^1 ([-L,L]) \, \colon \,  w(L) = - w(-L)=m \right\}$.
Assume that $v$ is increasing. Then, the continuous odd rearrangement
$v^t$ of $v$ with respect to $b$ satisfies:
\begin{enumerate}
\item[$($a$)$] If $b>0$ then
\begin{equation}\label{equidistributed}
\int_{-L}^L G(v^t) b(x)\, dx = \int_{-L}^L G(v)b(x) \, dx\quad \textrm{for all }0\leq t\leq 1.
\end{equation}
\item[$($b$)$] If $a\equiv b>0$ and ${\rm log}\, a$ is convex, then
the following assertions hold:
\item[$($b$.1)$] For all $0\leq t\leq 1$,
\begin{equation}\label{eq:PolyaSzego}
     \int_{-L}^L\left( \frac{ d v^t }{dx} \right)^2 \! \! a(x) \, dx
     \leq
     \int_{-L}^L \left( \frac{ d v }{dx} \right)^2 \!  \! a(x) \, dx.
\end{equation}
\item[$($b$.2)$] Equality in \eqref{eq:PolyaSzego} holds for some $t\in(0,1)$ if and only if
$v^t=v$ for all $t\in[0,1]$. In such case, $v$ must be odd.
\item[$($b$.3)$] The function $t\longmapsto\mathcal{E}(v^t,(-L,L))$ is convex in $[0,1]$.
\end{enumerate}
\end{Theorem}

Part $(a)$ will follow from the definition of continuous odd rearrangement. To prove parts
$(b.1)$ and $(b.2)$, we use the coarea formula to obtain
$$
   \int_{-L}^L  \left( \frac{d v^t}{dx } \right)^2 \! a(x) \, dx
   = \int_0^m \left\{\frac{a(\rho^t(\lambda))}{(\rho^t)'(\lambda)}
   +
   \frac{a(\rho^t(-\lambda))}{(\rho^t)'(-\lambda)}\right\}\ d\lambda,
$$
and then we compare the integrand in the second integral for $t\in(0,1)$ and
for $t=1$ using the log-convexity of $a$. Finally, part $(b.3)$ will follow
from $(a)$ and from differentiating twice the function $\mathcal{E}(v^t,(-L,L))$
and using that $a$ is log-convex, after regularizing $a$.

Theorem~\ref{thm:OddRearrangement} allows us to prove the following extension (in the one-dimensional case
and once we know that the solution is increasing)
of Berestycki and Nirenberg's result on antisymmetry (Theorem~\ref{thm:BerestyckiNirenberg}).
Our proof uses a completely different technique (rearrangement) than theirs
(the sliding method). Under the same hypothesis on the weight $a$, our method leads to antisymmetry
for increasing solutions not only for locally Lipschitz nonlinearities $f=-G'$ but also for discontinuous ones, since we only require $G$ to be locally Lipschitz.

\begin{Theorem}\label{thm:Brahms1}
Assume that $a\equiv b \in C^0([-L,L])$ and $G\in C^{0,1}_{\rm loc}(\mathbb{R})$ are even
functions and that $a>0$ in $[-L,L]$. Note that $f=-G'$ could be discontinuous.

Let $m >0$ and let $u \in H^1_m ((-L,L))$ be an increasing critical point of the
functional $\mathcal{E} (\cdot, (-L,L))$.
If $a$ is {\rm log}-convex, then $u$ is odd.
\end{Theorem}

Note that we assume that the critical point $u$ is increasing. For a {\rm log}-convex
weight $a\equiv b\in C^1([-L,L])$ and $G\in C^{1,1}_{\rm loc}(\mathbb{R})$, this automatically holds if either
$-m=u(-L)\leq u\leq u(L)=m$ in $(-L,L)$ or if
$G(-m)=G(m)\leq G$ in $\mathbb{R}$. That the first assumption
suffices is a consequence of Theorem~\ref{thm:BerestyckiNirenberg}, while the
sufficiency of the second one
---without requiring any \textit{a priori} estimate on the solution---follows from
Theorem~\ref{thm:Increasing}~(i) below. Note that the Allen-Cahn
potential $G (u) = (1 - u^2)^2/4$ with $m=1$ satisfies this last assumption.

\begin{Remark}\label{rem:OddRearrangement}
We will prove the statements in Theorems~\ref{thm:OddRearrangement} and
\ref{thm:Brahms1}, as well as the one stated in
Theorem~\ref{thm:BerestyckiNirenberg}, for general weights
$a$ and $b$ (not necessarily equal) ---see the beginning of the proof of
Theorem~\ref{thm:OddRearrangement} in Section~\ref{subsec2:2}. In particular,
they hold assuming that
\begin{equation}\label{eq:BerNiren}
\frac{\big( \sqrt{a b } \, \big)'}{b}\quad \textrm{is nondecreasing in }(-L,L)
\end{equation}
instead of the log-convexity of $a\equiv b$ (assuming $a,b\in C^1([-L,L])$). For this, see \eqref{newexp}
in the beginning of Section~\ref{subsec2:2}. Note that assumption \eqref{eq:BerNiren}
for $b\equiv1$ is equivalent to require that $\sqrt{a}$ is a convex function.
\end{Remark}

Property~\eqref{eq:PolyaSzego} has been first proved for the Steiner
or Schwarz symmetrization of a function with zero Dirichlet boundary data and
$a \equiv 1$ by P\'olya-Szeg\"o \cite{PolyaSzego} (see also \cite{Kawohl}).
Later, their well known result has been studied for non-constant
weights $a$ by several authors; see \cite{BBMP1999, BBMP2008, Brock1999,
Brock2000, EspositoTrombetta} among others.
In the higher dimensional case, Esposito and Trombetti~\cite{EspositoTrombetta} proved
that the functional
\begin{equation}\label{Trombetti:fct}
\int_{\mathbb{R}^n} \left\{\frac{1}{2}\tilde{a}(x')|\nabla_{x'} u|^2
+\frac{1}{2}a(x_N)u_{x_N}^2+\tilde{b}(x')G(u)\right\}\,dx,
\end{equation}
where $x=(x',x_N)\in\mathbb{R}^{N-1}\times\mathbb{R}$,
is decreased under Steiner symmetrization of functions $u$ with compact support
when $\sqrt a$ is strictly convex. Moreover, they proved that minimizers are Steiner
symmetric under this assumption.
Note that in the 1-dimensional case ($N=1$) this assumption coincides with the one in
Remark~\ref{rem:OddRearrangement}. Hence, our Theorem~\ref{thm:OddRearrangement}
shows that the same properties hold for our continuous odd rearrangement under the same
assumption as theirs.

Theorem~\ref{thm:OddRearrangement} may be easily extended to the $N$-dimensional
case (though we do not write the details in this paper). The result asserts that the
functional \eqref{Trombetti:fct} is decreased under the continuous odd rearrangement
with respect to the $x_N$ variable whenever $\sqrt{a(x_N)}$ is convex in $x_N$, $\tilde{b}\equiv 1$,
and $\tilde{a}$ is nonnegative.

For another rearrangement, the monotone decreasing one, Landes~\cite{Landes} shows
that~\eqref{eq:PolyaSzego} holds whenever $a$ is nonnegative and nondecreasing
(this result does not require any convexity assumption on $a$).

When $a\equiv b$, the log-convexity assumption in Theorem~\ref{thm:OddRearrangement}
also appears in a different (but related) context. In \cite{RCBM}, Rosales, Ca\~nete, Bayle,
and Morgan study the subsets $E$ of $\mathbb{R}$ (with given weighted volume $\int_Ea $)
which minimize the weighted perimeter $\int_{\partial E}a$. In Corollary~4.12 of \cite{RCBM}
they show that if $a$ is an even and strictly log-convex weight, then intervals centered at $0$
are the unique minimizers. In Section~\ref{section1:2} we will mention another result
of \cite{RCBM} in higher dimensions which is related to one of our results in dimensions
$N\geq 2$.

Paper \cite{RCBM} motivated the so called ``log-convex density conjecture'',
first stated by Kenneth Brakke, as follows. In $\mathbb{R}^N$, with a smooth,
radial, log-convex density, balls around the origin provide isoperimetric
regions of any given volume. The conjecture
has been recently proven by Gregory R. Chambers~\cite{Chambers}.

Odd symmetry of solutions may not hold without the previous monotonicity
or convexity-type assumptions on the weights.
In fact, for some weights and for the potential $G(u)=(1-u^2)^2/4$, we will
prove the existence of non-odd minimizers which are increasing from
$-1$ to $1$. Indeed, by considering the space of antisymmetric functions
$$
   H^{as}_m ((-L,L)) := \{ u \in H^1_m ((-L,L)) \, \colon \, u(x) = - u (-x) \}
   \,,
$$
we will provide sufficient conditions on the weights $a$ and $b$
for which
\begin{equation}\label{minleqmin}
   \min_{u \in H^{as}_m ((-L,L))} \mathcal{E} (u, (-L,L)) >
   \min_{u \in H^{1}_m ((-L,L))} \mathcal{E} (u, (-L,L))
\end{equation}
when $L$ is large enough.
Note that if this holds, then minimizers in $H^{1}_m ((-L,L))$
are not antisymmetric. The following conditions on the weights $a$ and $b$
guarantee \eqref{minleqmin} and therefore non-oddness of minimizers.
%
\begin{Proposition}\label{cor:characteriation}
Assume  that $a$, $b$, and $G$ are even $C^0(\mathbb{R})$ functions and that $a,b>0$. Let $m>0$
and suppose that $G(s) \geq G(m)$ for all $s \in
\mathbb R$ and $G(s)>G(m)$ for all $s\in(-m,m)$.

If there exists a sequence of bounded intervals $J_n
\subset \mathbb R$ satisfying
\begin{equation}\label{eq:Characterization}
  \int_{J_n}\frac{1}{a}\rightarrow +\infty
  \quad\textrm{and}\quad
  \int_{J_n} b \rightarrow 0\, ,
\end{equation}
then there exists $L_0>0$ such that $\mathcal{E} (\cdot,I)$ has no odd minimizers
on any interval $I:=(-L,L)$ with $L>L_0$.
\end{Proposition}

Note that this result applies to the Allen-Cahn potential and boundary
values~$\pm m=\pm1$.

In order to prove Proposition~\ref{cor:characteriation}, we can assume without
loss of generality that $G(m)=0$ by replacing $G$ by $G-G(m)$ if necessary. We will
first see that the infimum of the functional $\mathcal{E} (\cdot,I)$,
in the class of odd functions $H^{as}_m (I)$, is bounded from below by a positive
constant which is independent of the interval. Next, in Proposition~\ref{infimo=0} we prove that
condition~\eqref{eq:Characterization} is equivalent to the fact that
$$
    \min_{u \in H^1_m (I) } \mathcal{E} (u, I) \to 0
    \qquad
    \hbox{ as }  L \to \infty.
$$

\begin{Remark}\label{notunique}
Under the assumptions of Proposition~\ref{cor:characteriation},
we deduce that on any interval $I:= (-L,L)$ with $L > L_0$ the
functional $\mathcal{E} (\cdot, I)$ admits at least three critical
points:
\begin{enumerate}
\item[(i)]
Two minimizers: $u$ and its flipped $u_\star(x)=-u(-x)$ (which are different since $u$ is not antisymmetric).
\item[(ii)]
A critical point $u_{as}$ which is antisymmetric. It is obtained by minimizing the functional
$\mathcal{E} (\cdot, I)$ in the space $H^{as}_m (I)$.
\end{enumerate}
\end{Remark}

\begin{Example}\label{ex:nonodd}
Let us exhibit a simple class of weights for which critical points in large enough intervals
are not odd. Assume that $0<a,b\in C^0(\mathbb{R})$ and that $0\leq G\in C^1(\mathbb{R})$
are even functions, $a\in L^\infty (\mathbb{R})$, and
$\lim_{x\to+\infty}b(x)=0$.
Under these assumptions, take any sequence $x_n\to+\infty$ such that
$b(x)\leq 1/n^2$  for all $x>x_n$. Then, we have
$$
  \int_{x_n}^{x_n+n} \frac{1}{a} \, \geq \, \frac{n}{\Vert a\Vert_\infty}\to +\infty
  \quad \hbox{ and } \quad
  \int_{x_n}^{x_n+n} b
  \, \leq \,  n\frac{1}{n^2}\to 0.
$$
Thus, condition~\eqref{eq:Characterization} is satisfied. Therefore, by
Proposition~\ref{cor:characteriation}, there exists $L_0>0$ such that $\mathcal{E}$
has no odd minimizers on $(-L,L)$ whenever $L>L_0$.
\end{Example}

A related, but different, question is the existence of non-odd minimizers
in $H^1_m((-L,L))$ which are not increasing. For $G(u) = (1-u^2)^2/4$ this cannot
happen if $u(\pm L)=\pm 1$ (in this case any minimizer is increasing), but
it may occur if $u(\pm L)=\pm\varepsilon$ with $\varepsilon$ small and $L$
large. In Proposition~\ref{prop:NonOdd} below we prove the existence of such
non-odd minimizers of $\mathcal{E}$ for a large family of weights, which includes the
unweighted case $a\equiv b\equiv 1$. This result will be proved using a
perturbation argument from the case $m=0$ (see Figure \ref{fig2}).

\begin{figure}[ht]
\begin{center}
\psset{xunit=6cm,yunit=3cm}
\begin{pspicture}(0,1.2)
\psaxes[linestyle=dashed,axesstyle=frame,ticks=none]{->}(0,0)(-1,-1)(1,1)
\psline[linewidth=0.5pt]{->}(-1.1,0)(1.1,0)
\psline[linewidth=0.5pt]{->}(0,-1.1)(0,1.1)
\psline[linecolor=blue,linestyle=dashed,linewidth=0.5pt](-1,-0.2)(0,-0.2)
\psline[linecolor=blue,linestyle=dashed,linewidth=0.5pt](0,0.2)(1,0.2)
\pscurve[linecolor=red,linewidth=0.5pt]{*-*}(-1,0)(0,0.8)(1,0)
\pscurve[linecolor=blue,linewidth=0.5pt]{*-*}(-1,-0.2)(0,0.6)(1,0.2)
\rput(-0.58,0.55){$u_0$}
\rput(-0.4,0.25){$u_\varepsilon$}
\rput(1.15,0){$x$}
\rput(0,1.15){$\lambda$}
\rput[linecolor=blue](0.01,-0.2){$-\varepsilon$}
\rput[linecolor=blue](-0.05,0.2){$\varepsilon$}
\rput(1.05,-0.1){$L_0$}
\rput(-1.08,-0.1){$-L_0$}
\end{pspicture}
\end{center}
\vspace{2.5cm}
  \caption{Minimizers for $m=0$ and $m=\varepsilon$: $u_0$ and $u_\varepsilon$}
  \label{fig2}
\end{figure}
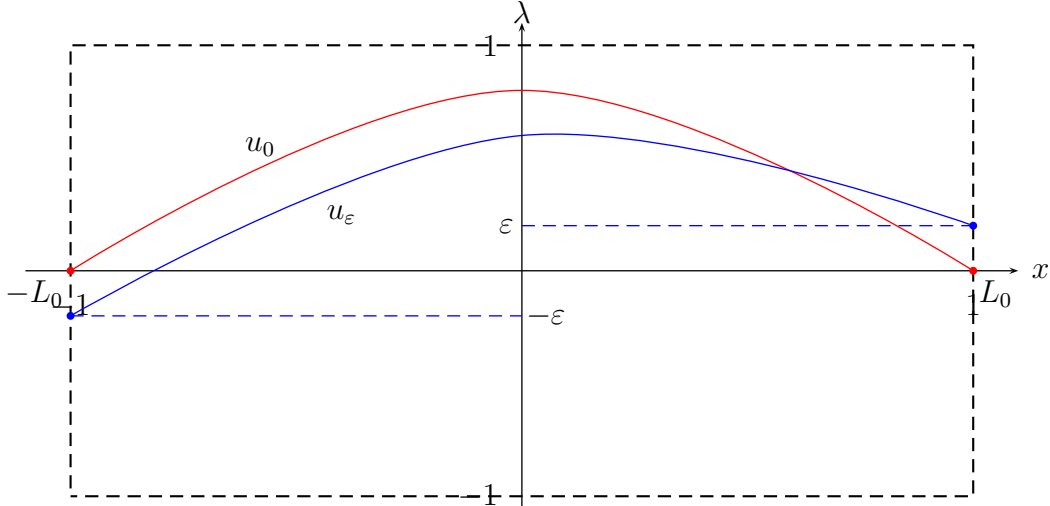

Another contribution of our paper is to provide conditions on the
weights $a$, $b$, and on the potential $G$ which guarantee the monotonicity
of solutions for the one-dimensional problem~\eqref{E-L},
without relying on the {\it a priori} bound~\eqref{apriori} used
in~\cite{BeresNiren}. Recall that in Theorem~\ref{thm:Brahms2} we already
gave conditions to guarantee uniqueness and monotonicity of
solutions. Our following result guarantees monotonicity under more
general conditions on $a$, $b$, and $G$. Here $a$, $b$, and $G$ are not
assumed to be even. In the even case, we would take $x_0=0$ in the
following condition \eqref{eq:Muffin}.
%
\begin{Theorem} \label{thm:Increasing}
Let $a,b \in C^1 ([-L,L])$ such that $a,b>0$, and $G\in C^{1,1}_{\rm loc}(\mathbb{R})$.
Assume that there exists $x_0 \in [-L,L]$ such that
\begin{equation} \label{eq:Muffin}
  (ab)' \leq 0
  \, \hbox{ in } (-L,x_0]
  \quad \hbox { and } \quad
  (ab)' \geq 0
  \, \hbox{ in } [x_0,L).
\end{equation}
Then, any solution to \eqref{E-L} with $m>0$ is increasing if
either
\begin{enumerate}
\item[{\rm (i)}]
$G\geq G(-m)=G(m)$ in $\mathbb{R}$;
\end{enumerate}
or
\begin{enumerate}
\item[{\rm (ii)}]
For some $M\in(0,m]$ the function $G$ satisfies
\begin{equation} \label{eq:Double:bis}
\begin{array}{l}
  G\geq G(-M)=G(M)\textrm{ in }[-M,M], \quad
  G' \leq 0 \textrm{ in } (-\infty,- M),
\\
  \textrm{and}\quad
  G' \geq 0 \textrm{ in } (M,+\infty).
\end{array}
\end{equation}
\end{enumerate}
\end{Theorem}
%

As an example, note that the Allen-Cahn potential $G (u) = (1 - u^2)^2/4$ with $M=1$ satisfies
assumption \eqref{eq:Double:bis}. In this particular case, Theorem~\ref{thm:Increasing}~(ii)
establishes that any solution is increasing if $m\geq 1$. Instead, as we said before,
in the case where $m=\varepsilon<1$ solutions which are not increasing do exist
(see Figure~\ref{fig2} and Proposition~\ref{prop:NonOdd}).

Note that, when $a$ and $b$ are even, the monotonicity condition \eqref{eq:Muffin} is
weaker than the convexity-type assumption \eqref{eq:BerNiren} (see
Remark~\ref{rem:MonAme}) and that we do not assume any {\it a priori}
bound on the solution.
As a consequence, if the weights $a$ and $b$ satisfy \eqref{eq:BerNiren}
then any solution $u$ to \eqref{E-L} is increasing under assumption (i)
or (ii) of Theorem~\ref{thm:Increasing}, and hence, the \textit{a priori}
estimate \eqref{apriori} automatically holds.

To prove Theorem~\ref{thm:Increasing} we use the ``Hamiltonian'' $\mathcal{H}$
defined in \eqref{L:intro} and a phase plane type analysis.

\subsection{The higher dimensional case}\label{section1:2}
In the remaining of the Introduction, we consider the extension of the
functional~\eqref{functional} to a $N$-dimensional domain. More
specifically, given a bounded domain $\Omega \subset \mathbb R^N$, a
$C^1$-map $A: \overline{\Omega} \to S_N (\mathbb R)$ with range in the set of
symmetric matrices and assumed to be uniformly coercive, a function $0<b \in C^1
(\overline{\Omega}, \mathbb R)$, and a potential $G\in C^2(\mathbb{R})$ satisfying
that
\begin{equation}\label{newG}
\begin{array}{l}
\textrm{there exists }M>0\textrm{ such that } G'(s)\leq 0\textrm{ for all }s< -M
\\
\textrm{and }G'(s)\geq 0\textrm{ for all }s>M,
\end{array}
\end{equation}
we consider the functional
\begin{equation}\label{eq:NDimFunctional}
   \mathcal{E} (u, \Omega)
   :=
   \int_{\Omega}
   \left\{  \frac{1}{2} \langle A(x) \nabla u,  \nabla u \rangle  + b(x) G(u) \right\} dx ,
   \quad
   u \in H^{1}_{\varphi} (\Omega) ,
\end{equation}
where $\varphi  \in (H^1\cap L^\infty)(\Omega)$ and
$$
  H^{1}_{\varphi} (\Omega) :=
  \left\{ u \in H^1 (\Omega) \, \colon \, u - \varphi \in H^1_0 (\Omega )
  \right\} .
$$
The Euler-Lagrange equation associated to \eqref{eq:NDimFunctional} is
given by
\begin{equation} \label{eq:Giessen}
   - \hbox{div} (A(x) \nabla u)  + b(x) G'(u) = 0,
   \quad
   u \in H^1_{\varphi} (\Omega).
\end{equation}

Recall that under quite restrictive assumptions on $A$ and $b$, a result
for the odd rearrangement in the $x_N$-variable
(which leads to antisymmetry) has been mentioned in \eqref{Trombetti:fct} and
the comments after it. This result requires assumptions on how $A$ and $b$ depend on the
variables $x=(x',x_N)$.

In the following (and without the previous restrictive assumptions),
we will prove several uniqueness results.
Setting
$$
  \lambda_1 (A,b, \Omega) :=
  \inf\left\{ \frac{\int_{\Omega} \langle A(x) \nabla \xi, \nabla \xi \rangle}
                   {\int_{\Omega} b(x) \xi^2}
       \, \colon \, \xi \in H^1_{0} (\Omega), \, \xi \not \equiv 0
  \right\} ,
$$
under the assumption that
\begin{equation} \label{SturmUndDrang}
    \lambda_1 (A,b,\Omega) \geq - G''(0) > - G''(s)
    \quad
    \textrm{for all } s \not = 0,
\end{equation}
simple arguments show that the functional $\mathcal{E} (\cdot, \Omega)$ has a unique critical
point in  $H^1_{\varphi} (\Omega)$ (see Proposition~\ref{Proposition:Marc1}).

For the double-well potential
$G(s) = \frac{1}{4}(1-s^2)^2$, namely the type of nonlinearity arising in the
De Giorgi conjecture discussed above, the condition $-G''(0) > -G''(s)$ for all $s\neq 0$
(as well as \eqref{newG})
is clearly satisfied. It is also easy to verify that $ \lambda_1 (A,b,\Omega) \geq - G''(0) $
holds for small domains $\Omega$.
Therefore it is of interest to find a class of weights for which this lower bound
is independent of the domain. A typical situation for which this holds is provided
by the weights $A(x) = e^{\alpha |x|^2} {\rm Id}$,
$b(x) = e^{\alpha|x|^2}$ with $\alpha$ large enough.
For a more general class of weights, we are able to give an explicit lower bound on $ \lambda_1 (A,b,\Omega)$
depending on $(A,b)$, which in the simplest case $A \equiv a\, {\rm Id}$ and $a \equiv b$
leads to the following uniqueness result:

\begin{Theorem} \label{thm:ndim}
Let $a\in C^2(\overline{\Omega})$, $G \in C^2 (\mathbb R)$ satisfying \eqref{newG}, and
$\varphi\in (H^1\cap L^\infty)(\Omega)$.
Assume $A \equiv a \, {\rm Id}$, $a\equiv b>0$ in $\overline\Omega$,
$$
  \frac{\Delta \sqrt{a}}{\sqrt{a}}  \geq - G''(0) \hbox{ in } \Omega,
  \quad \hbox{ and } \quad
  -G''(0) > -G''(s)
  \textrm{ for all } s \neq 0.
$$
Then $\mathcal{E} (\cdot, \Omega)$ admits a unique critical point in $H^1_\varphi(\Omega)$.
\end{Theorem}
%
As a consequence we obtain the following result.
Let $\sigma:\mathbb R^N \to \mathbb R^N$ be a reflection with respect
to a hyperplane. If
\begin{equation} \label{eq:Flughafen}
  G(s) = G(-s)\textrm{ for all }s\in\mathbb{R},
  \quad
  A \circ \sigma = A,
  \quad
  b \circ \sigma = b ,
  \quad
  \varphi \circ \sigma =  - \varphi,
\end{equation}
and $\sigma$ leaves $\Omega$ invariant (\textit{i.e.}, $\sigma(\Omega) = \Omega$),
then the critical points of $\mathcal{E} (\cdot, \Omega)$ in $H^1_\varphi(\Omega)$
inherit this same invariance:
%
\begin{Corollary} \label{cor:OddHigher}
Assume that $\sigma(\Omega)=\Omega$ and that condition \eqref{eq:Flughafen} holds
for some reflection $\sigma:\mathbb R^N \to \mathbb R^N$ with respect to a hyperplane.
Then, under the hypotheses of Theorem~{\rm \ref{thm:ndim}}, for every $\varphi\in (H^1\cap L^\infty)(\Omega)$
the functional $\mathcal{E} (\cdot, \Omega)$ in \eqref{eq:NDimFunctional}
admits a unique critical point $u$ in $H^1_\varphi(\Omega)$. In particular, $u$ is
antisymmetric, in the sense that $u \circ \sigma = - u$.
\end{Corollary}

\begin{Remark}
The conclusions on uniqueness and antisymmetry of Theorem~\ref{thm:ndim} and
Corollary~\ref{cor:OddHigher} hold in the two following cases in which we assume
$A\equiv a\,{\rm Id}$ and $G(u)=(1-u^2)^2/4$.
\begin{enumerate}
\item[{\rm (i)}] $a(x)=b(x)=e^{\alpha|x|^2}$ for all $x\in\Omega$ and $2\alpha N\geq 1$
(see Example~\ref{ex6:2}~(ii)).
\item[{\rm (ii)}]$a(x)=|x|^{\beta+2}$, $b(x)=|x|^\beta$ for all $x\in\Omega$,
and $\beta > -N$ (see Remark~\ref{rem6:7}. Note that here $a\not\equiv b$).
\end{enumerate}
The log-convex weight $e^{\alpha|x|^2}$, $\alpha>0$, also appears in the paper \cite{RCBM}.
There, in Theorem 5.2, it is proved that balls in $\mathbb{R}^N$ centered at the origin
are the unique minimizers of weighted perimeter for a given weighted volume.
For this, the authors use Steiner symmetrization among other tools.
\end{Remark}

Note that the saddle-shaped solution (mentioned above in the De Giorgi
conjecture in $\mathbb{R}^{2m}$) is a function $u=u(s,t)$ of two radial variables
and it is a critical point of the functional
$$
   \int_{(0,L)^2} \Big\{ \frac{1}{2} |\nabla u|^2 + G(u)  \Big\} s^{m-1} t^{m-1} ds dt,
$$
namely a functional of the type \eqref{eq:NDimFunctional} with
$A(x) = (st)^{m-1} {\rm Id}$ and $b(x) = (st)^{m-1}$.
For these weights, our results show that the minimizers are antisymmetric if $L$ is small enough,
whereas our uniqueness result cannot be applied for large $L$ (see Section~6).

\subsection{Plan of the paper}
We have organized our paper as follows. The continuous odd rearrangement
and its main properties, stated in Theorems~\ref{thm:OddRearrangement},
\ref{thm:Brahms1}, and Remark~\ref{rem:OddRearrangement}, are contained in
Section~\ref{sec:Rearrangement}.
In Section~\ref{sec:Monotonicity}, we discuss the monotonicity of
solutions for the one-dimensional problem, and prove
Theorem~\ref{thm:Increasing}.
In Section~\ref{section4}, we prove our uniqueness result stated
in Theorem~\ref{thm:Brahms2}, as well as a more general result
(Corollary~\ref{thm:Uniqueness}).
Section~\ref{section5} is devoted to give conditions on the weights
under which minimizers in large intervals are not odd functions (we
prove in particular Proposition~\ref{cor:characteriation}).
Finally, in Section~\ref{sec:ndim} we give some uniqueness results
in higher dimensions and prove Theorem~\ref{thm:ndim}.

\section{Antisymmetry of critical points: continuous odd rearrangement} \label{sec:Rearrangement}

In this section we collect general properties of minimizers and
we show how antisymmetry of critical points can be obtained
by using our new continuous rearrangement.

\subsection{General properties of minimizers}

We start by giving some qualitative properties of minimizers of
$\mathcal{E} (\cdot, I)$ that can be obtained without any monotonicity
assumption on the weights $a$ and $b$. Here, and in the rest of the paper,
$$
I:=(-L,L).
$$

\begin{Lemma}\label{lem:OderedMin}
Let $L>0$, $m\geq 0$, $G\in C^{1,1}_{\rm loc}(\mathbb{R})$, and assume that $a$, $b$, and $G$ satisfy \eqref{eq:A}.
If $u_1$ and $u_2$ are two minimizers of $\mathcal{E} (\cdot,I)$ in $H^1_m(I)$,
then the following alternative holds:
\begin{equation*} \label{eq:OderedMin}
  \textrm{either }u_1 > u_2
  \hbox { in }  I,
  \quad \, \textrm{ or }
  u_1 < u_2
  \hbox{ in } I,
  \quad \, \textrm{ or }
  u_1 \equiv u_2
  \hbox{ in } I.
\end{equation*}
\end{Lemma}

Before commenting the proof of the lemma, let us start with some generalities that will be used
at different moments of the paper. First, a minimizer $u$ as in the lemma will be a $C^2$ function
satisfying \eqref{E-L} pointwise. Indeed, $u$ being in $H^1_m(I)$ tells us that $u$ is continuous
in $[-L,L]$. Thus $bf(u)$ is also continuous and, by the weak sense of \eqref{E-L}, $au'$
will be $C^1$. Since $a>0$ is $C^1$,
we conclude that $u\in C^2((-L,L))$.

Second, under the hypotheses of the lemma (in particular,  $G\in C^{1,1}_{\rm loc}(\mathbb{R})$),
the initial value problem for the ODE $-(au')'=bf(u)$ in \eqref{E-L} enjoys uniqueness.
More precisely, if two solutions $u_1$
and $u_2$ of the ODE satisfy $u_1(x_0)=u_2(x_0)$ and $u_1'(x_0)=u_2'(x_0)$ for some $x_0\in (-L,L)$,
then they agree. This is a consequence of the classical uniqueness theorem for ODEs, which
in our case requires $a$, $a'$, and $b$ to be bounded and continuous, and $f=-G'$ to be Lipschitz continuous.

\begin{proof}[Proof of Lemma~{\rm \ref{lem:OderedMin}}]
We use a well known cutting and energy argument. One considers the function $v:=\min(u_1,u_2)$,
which satisfies $v\leq u_1$. Using that both $u_1$ and $u_2$ are minimizers, one easily
shows that $v$ has the same energy as $u_1$, and thus $v$ is also a minimizer (see
the details in~\cite[Lemma 3.1]{JerisonMonneau}, for instance).

Now, since $v\leq u_1$ are both solutions of the equation, the strong maximum principle leads
to the alternative of the lemma. Alternatively, the same conclusion can be deduced from the uniqueness
theorem for the initial value problem for the ODE (commented above).
\end{proof}

The following proposition collects other important properties of minimizers.
%
\begin{Proposition} \label{prop:Schubert}
Let $L>0$, $m\geq 0$, $G\in C^{1,1}_{\rm loc}(\mathbb{R})$, and that $a$, $b$, and $G$ satisfy \eqref{eq:A}.
If $u$ is a minimizer of $\mathcal{E} (\cdot,I)$ in $H^1_m(I)$,
then the following hold:
\begin{enumerate}
\item[{\rm (i)}]
For $m=0$, we have either $u \equiv 0$ or $|u|>0$ in $I$. Moreover, $u$ is even;
\item[{\rm (ii)}]
$u$ is odd if and only if  $u(0) = 0$;
\item[{\rm (iii)}]
 For $m >0$, $u$ has exactly one zero and $u'(0) >0$;
\item[{\rm (iv)}]
If $m>0$ and $ -m \leq u \leq m$, then $u$ is increasing.
\end{enumerate}
\end{Proposition}
\begin{proof}
{\bf (i)} If $m=0$, the assumption that $G$ is even gives that
$|u| \in H^1_0 (I)$ is also a minimizer. A classical argument based
on the strong maximum principle immediately yields the alternative
$u \equiv 0$ or $|u| >0$ in $I$.

Consider $v(x) := u(-x)$, which is also solution of \eqref{E-L} since $m=0$.
We easily check $\mathcal{E} (u,I) = \mathcal{E} (v,I)$, which shows that $v$ is also a minimizer. Since
$u(0) = v(0)$, we must have $u \equiv v$ by Lemma~\ref{lem:OderedMin}.

{\bf (ii)}
Assume $u(0)=0$. Without loss of generality, we may assume that
$$
  \int_0^L \left\{ \frac{u'^2}{2} a(x)  + G(u) b(x) \right\} dx
  \geq
  \int_{-L}^0 \left\{\frac{u'^2}{2} a(x) + G(u) b(x) \right\} dx.
$$
By defining
$$
   {\tilde u} (x) :=
   \left\{
   \begin{array}{ll}
      u (x)    &\hbox{ for } x \in (-L,0) , \\
      -u(-x)   &\hbox{ for } x \in (0,L),
   \end{array}
   \right.
$$
we easily see  that ${\tilde u}\in H^1_m(I)$ and
$\mathcal{E} (u, I) \geq   \mathcal{E} ({\tilde u}, I)$. Hence ${\tilde u}$ is an odd
minimizer satisfying ${\tilde u}(0) = u(0)$.
By the alternative of Lemma~\ref{lem:OderedMin} we deduce that ${\tilde u} \equiv u$.
This shows that $u$ is odd.

Note that this argument also works in some higher dimensional case under the assumption
that $u$ vanishes in a hyperplane as well as assuming appropriate symmetry assumptions on the domain,
the boundary condition,  and the weights. In dimension
one there is another proof that gives the same statement not only for minimizers
but also for solutions of \eqref{E-L}. Indeed, let $u$ be a solution of \eqref{E-L}
such that $u(0)=0$ and let $u_\star$ be its flipped (as in \eqref{def:flipped}).
Since $u$ and $u_\star$ are solutions of \eqref{E-L}
satisfying $u(0)=u_\star$(0) and $u'(0)=u_\star'(0)$, we conclude that
$u=u_\star$ by uniqueness for the Cauchy problem. Therefore, $u$ is odd.

{\bf (iii)} Let $m>0$.
Assume by contradiction that there exists a nonempty interval $(x_1,x_2) \subset I$
such that $u(x_1) = u(x_2) = 0$ and $u>0$ in $(x_1,x_2)$.
Let $\tilde{u}:=-u$ in $(x_1,x_2)$ and $\tilde{u}:=u$ in $\overline{I}\setminus(x_1,x_2)$
and note that $\mathcal{E} (u, I) =  \mathcal{E} ({\tilde u}, I)$ by \eqref{eq:A}.
Using the alternative of Lemma~\ref{lem:OderedMin} we have a contradiction, since we
would have two minimizers $u$ and $\tilde{u}$ satisfying $u\equiv \tilde{u}$
in $\overline{I}\setminus(x_1,x_2)$.

Consider $v(x) := u(-x)$. We claim that $u>v$ in $(0,L)$. Note that $u(0)=v(0)$
and $u(L)=m>-m=v(L)$. Indeed, assume first that $u<v$ in an open interval
$J\subset(0,L)$ and $u=v$ on $\partial J$. Replacing $u$ by $v$ if necessary
we may assume that $\mathcal{E}(v,J)\leq \mathcal{E}(u,J)$. Therefore, defining $\bar{u}=u$ in
$I\setminus J$ and $\bar{u}=v$ in $\bar{J}$, we obtain that $\bar{u}$ is a minimizer
of $\mathcal{E}(\cdot,I)$ in $H^1_m(I)$ different from $u$. This is a contradiction
with the alternative of Lemma~\ref{lem:OderedMin} and proves that $u\geq v$ in $(0,L)$.
However, if there exists $x_0\in(0,L)$ such that $u(x_0)=v(x_0)$ we would have
$u'(x_0)=v'(x_0)$ (since $u\geq v$ in $(0,L)$). This is a contradiction we the
uniqueness of the Cauchy problem, since $u$ and $v$ are solutions of equation
\eqref{E-L} satisfying $u(x_0)=v(x_0)$, $u'(x_0)=v'(x_0)$, and $u(L) \not = v(L)$,
and proves the claim.

As a consequence, we obtain that $u'(0) \geq 0$, and in fact, $u'(0) > 0$ (otherwise we would
have $u \equiv v$, again by uniqueness of the Cauchy problem, which cannot hold since $u(L) \not = v(L)$).

{\bf (iv)}
We first claim that $u$ is nondecreasing.
Assume on the contrary that the minimizer admits two local extrema. Let
$x_1\in(-L,L)$ be the smallest local maximum and let $s_1\in(-m,m]$ be its critical value. Let
$\bar{x}$ be the smallest solution to $u(x)=s_1$ in $(x_1,L]$ and
$s_2=\min_{[x_1,\bar{x}]}u$. Let  $G(\bar{s}) = \min_{s \in [s_2,s_1]} G(s)$,
$\bar{x}_1=\sup\{\tau\in(-L,x_1):u(\tau)=\bar{s}\}$, and
$\bar{x}_2=\sup\{\tau\in(-L,\bar{x}):u(\tau)=\bar{s}\}$ (see Figure~\ref{fig3}).

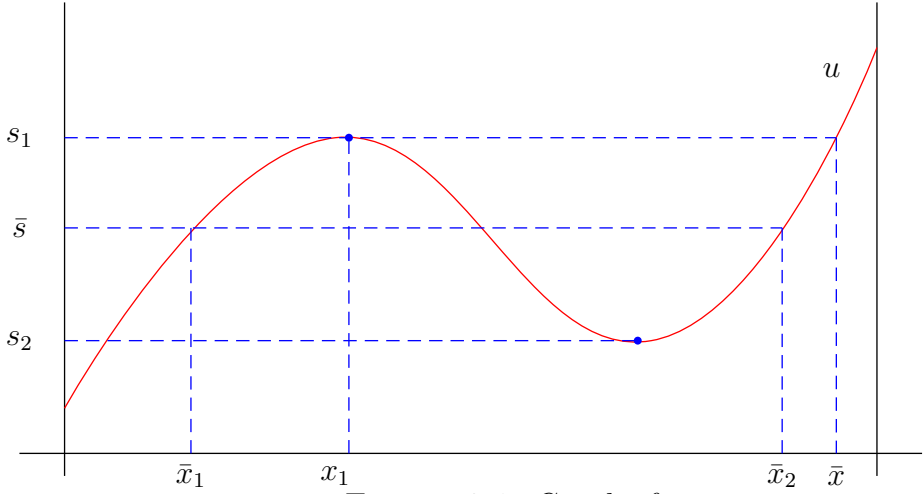
\begin{figure}[ht]
\begin{center}
\psset{xunit=6cm,yunit=3cm}
\begin{pspicture}(0,1.1)
\psline[linewidth=0.5pt]{-}(-1.1,-1)(0.9,-1)
\psline[linewidth=0.5pt]{-}(-1,-1.1)(-1,1)
\psline[linewidth=0.5pt]{-}(0.8,-1.1)(0.8,1)

\pscurve[linecolor=red,linewidth=0.5pt]{-}(-1,-0.8)(-0.4,0.4)(0.3,-0.5)(0.8,0.8)

\psline[linecolor=blue,linestyle=dashed,linewidth=0.5pt]{-*}(-0.37,-1)(-0.37,0.4)
\rput(-0.4,-1.1){$x_1$}

\psline[linecolor=blue,linestyle=dashed,linewidth=0.5pt]{-}(-1,0.4)(0.71,0.4)
\psline[linecolor=blue,linestyle=dashed,linewidth=0.5pt]{-}(0.71,0.4)(0.71,-1)
\rput(-1.1,0.4){$s_1$}
\rput(0.71,-1.1){$\bar{x}$}

\psline[linecolor=blue,linestyle=dashed,linewidth=0.5pt]{-*}(-1,-0.5)(0.27,-0.5)
\rput(-1.1,-0.5){$s_2$}

\psline[linecolor=blue,linestyle=dashed,linewidth=0.5pt]{-}(-1,0.0)(0.59,0.0)
\psline[linecolor=blue,linestyle=dashed,linewidth=0.5pt]{-}(-0.72,-1.0)(-0.72,0.0)
\psline[linecolor=blue,linestyle=dashed,linewidth=0.5pt]{-}(0.59,-1.0)(0.59,0.0)
\rput(-1.1,-0.0){$\bar{s}$}
\rput(-0.72,-1.1){$\bar{x}_1$}
\rput(0.59,-1.1){$\bar{x}_2$}

\rput(0.7,0.7){$u$}

\end{pspicture}
\end{center}
\vspace{2.7cm}
  \caption{Graph of $u$}
  \label{fig3}
\end{figure}

Defining
$$
      \tilde u (x) :=
   \left\{
      \begin{array}{cl}
          \bar{s}     &\hbox{ if } x \in (\bar{x}_1,\bar{x}_2), \\
           u(x)       &\hbox{ otherwise},
      \end{array}
   \right.
$$
we easily see that $\tilde{u}\in H^1_m(I)$ and $\mathcal{E} (\tilde u, I) \leq \mathcal{E} (u, I) $.
Therefore, $\tilde{u}$ is a minimizer which is constant in $[\bar{x}_1,\bar{x}_2]$.
This contradiction proves the claim.

Hence, ${\tilde u}$ is a nondecreasing minimizer and by the strong maximum
principle it must be increasing.
\end{proof}

\subsection{Continuous odd rearrangement}\label{subsec2:2}

In this subsection we prove Theorems \ref{thm:OddRearrangement} and \ref{thm:Brahms1}.
As we said in Remark~\ref{rem:OddRearrangement} both results will be proved in fact for functionals
for which the weights $a$ and $b$ are not necessarily equal.

Given an increasing and odd diffeomorphism $\gamma: (-L,L) \to (-{\tilde L}, {\tilde L})$
of class $C^1$ and making the change of variables $x=\gamma^{-1}(y)$ we obtain
$$
\mathcal{E}(u,(-L,L))=\int_{-L}^L\frac{1}{2}(u')^2a(x)+G(u)b(x)\,dx
=\int_{\gamma(-L)}^{\gamma(L)}\frac{1}{2}(\tilde{u}')^2\tilde{a}(y)+G(\tilde{u})\tilde{b}(y)\,dy
$$
where
$$
   {\tilde u} := u \circ \gamma^{-1},
   \qquad
   {\tilde a} := (a \gamma') \circ \gamma^{-1},
   \qquad\textrm{and}\qquad
   {\tilde b} :=  \frac{b}{\gamma'} \circ \gamma^{-1} \,.
$$
Similarly, a straightforward computation shows that, when $G\in C^{0,1}_{\rm loc}(\mathbb{R})$,
problem~\eqref{E-L} is equivalent to the following
\begin{equation}\label{eq:E-LBis}
\left\{
\begin{array}{l}
 - ( \tilde{a} {\tilde u} ')'={\tilde b} \, f({\tilde u})
   \quad
   \hbox{ in } (-\tilde{L},\tilde{L}),\\
  {\tilde u} (\tilde{L}) \, = \, - {\tilde u} (- \tilde{L})= m,
\end{array}
\right.
\end{equation}
where $\tilde L=\gamma(L)$.

In particular, the diffeomorphism
\begin{equation} \label{eq:ChangeVariable}
   \gamma_1 (x) :=  \int_0^x \sqrt{\frac{b}{a} }
   \qquad (\hbox{respectively,} \, \, \,
   \gamma_2 (x) := \int_0^x b)
\end{equation}
allows to rewrite the functional $\mathcal{E}(u,(-L,L))$ (or problem~\eqref{E-L}) as
$$
\mathcal{E}(\tilde{u},(-\gamma(L),\gamma(L)))
= \int_{-\gamma(L)}^{\gamma(L)}\frac{1}{2}(\tilde{u}')^2\tilde{a}(y)+G(\tilde{u})\tilde{b}(y)\,dy
$$
(or ~\eqref{eq:E-LBis}) with weights $({\tilde a}, {\tilde b})$ satisfying
${\tilde a} \equiv {\tilde b}=\sqrt{ab}\circ\gamma_1^{-1}$
(respectively, $\tilde{a}=(ab)\circ\gamma_2^{-1}$ and ${\tilde b} \equiv 1$).

Note that
\begin{equation}\label{newexp}
\frac{(\sqrt{ab})'}{b}\circ \gamma_1^{-1}
=
\frac{(\sqrt{ab})'\circ\gamma_1^{-1}}{\sqrt{ab}\circ\gamma_1^{-1}}(\gamma_1^{-1})'
=
\frac{\tilde{a}'}{\tilde{a}}
=
({\rm log}\, \tilde{a})'
\end{equation}
and
\begin{equation}\label{newexp2}
\frac{(\sqrt{ab})'}{b}\circ\gamma_2^{-1}
=
(\sqrt{ab}\circ\gamma_2^{-1})'
=
(\sqrt{\tilde{a}})'.
\end{equation}
This shows that assumption \eqref{eq:BerNiren} is equivalent to the log-convexity of $\tilde{a}$
when ${\tilde a} \equiv {\tilde b}=\sqrt{ab}\circ\gamma_1^{-1}$ and to the convexity of $\sqrt{\tilde{a}}$
when $\tilde{a}=(ab)\circ\gamma_2^{-1}$ and ${\tilde b} \equiv 1$.

We now define the continuous odd rearrangement of an increasing function
(with respect to the weight $b\equiv1$).

\begin{Definition}\label{def:rear}
Let $v \in H^1_m (I)$ be an increasing function. Let us denote the inverse of $v$ as $\rho$, \textit{i.e.},
$$
v(x)=\lambda\qquad\textrm{if and only if}\qquad \rho(\lambda)=x.
$$
Let $t\in[0,1]$ and define the family of functions
$$
\rho^t(\lambda):=t\rho(\lambda)+(1-t)\rho_\star(\lambda)\quad\textrm{ for all }
\lambda\in [-m,m],
$$
where $\rho_\star(\lambda)=-\rho(-\lambda)$ denotes the flipped of $\rho$.
It is clear that $\rho^t$ is an increasing function for all $t\in[0,1]$.
We define the \textit{continuous odd rearrangement} $\{v^t\}_{0\leq t\leq1}$
of $v$ as the family of inverse functions of $\{\rho^t\}_{0\leq t\leq1}$.
\end{Definition}

\begin{Remark}\label{rmk:lll}
Note that $\rho^t$ (or $v^t$) will be an odd function if and only if
$(2t-1)(\rho(\lambda)+\rho(-\lambda))=0$ for all $\lambda\in[-m,m]$.
In particular, the continuous rearrangement $v^t$ is an odd function
if either $\rho$ is odd (\textit{i.e.}, $v$ is odd) or $t=1/2$.
We call $v^{1/2}$ the \textit{odd rearrangement} of $v$.
\end{Remark}

For a positive even weight $a$ which is square root convex (when $b\equiv1$), and a general
even nonlinearity $G$, we can prove that continuous odd rearrangements and
Schwarz rearrangements share similar properties.

We start proving that the weighted Dirichlet integral is decreased under the odd rearrangement
$v^{1/2}$ when $\sqrt{a}$ is convex. This is the key that later leads to oddness of minimizers.
However, since we also prove oddness of critical points (not necessarily minimizers), we need to
use the whole family $v^t$, $t\in[0,1]$, in the continuous odd rearrangement.
Next result states that all functions $v^t$ have less weighted Dirichlet energy than $v$, when
$\sqrt{a}$ is convex.
\begin{Lemma}\label{Lem:Dirichlet}
Let $a \in C^0([-L,L])$ be an even positive function such that
$\sqrt a$ is convex. If $v \in C^1 ([-L,L])$ is increasing and $v(L)=-v(-L)=m>0$,
then it holds either that
\begin{equation}\label{claim1}
h(t):=\int_{-L}^L \left( \frac{d v^t}{dx } \right)^2 \! a(x) \, dx
  \, < \, \int_{-L}^L
  \left( \frac{d v}{dx } \right)^2 \! a(x) \, dx   \quad \textrm{for all }t\in(0,1),
\end{equation}
or that $v^t=v$ for all $t\in[0,1]$.
\end{Lemma}
Valenti~\cite{NOI} proved that $h(1/2) \leq h(1)$ holds even for $H^1_m$
functions. Its proof uses a reflection argument to convert the odd symmetry property into
even symmetry. He then uses Schwarz decreasing symmetrization ---which
applies to $H^1$ functions (see \cite{LL}). Next we provide a different proof than
the one given in \cite{NOI}.
In fact, our proof applies to all $v^t$, $t\in(0,1)$, assuming that $v\in C^1$.
A possible way to prove our lemma for functions $v\in H^1_m$ (that we do not do here)
would be extending to $v^t$, $t\in(0,1)$,
Coron's result~\cite{Coron} on the continuity of Schwarz rearrangement in dimension $1$
in the $W^{1,p}$-norm. This surely works for $t=1/2$, by the results of \cite{NOI}. Note that Coron's
result is a delicate one and, in fact, continuity of Schwarz rearrangement in $H^1$-norm
does not hold in higher dimensions $N\geq 2$ (see \cite{Alm-Lieb}).

We now establish Lemma~{\rm \ref{Lem:Dirichlet}}.
An alternative proof is given below (see Remark~\ref{Rmk:26}) as a consequence of Lemma~\ref{Lem:Convex_energy}.
The proof is shorter but requires the use of the whole family $v^t$, $t\in(0,1)$, even to establish
Lemma~\ref{Lem:Dirichlet} for $t=1/2$.

\begin{proof}[Proof of Lemma~{\rm \ref{Lem:Dirichlet}}]
We first establish that \eqref{claim1} with $<$ replaced by $\leq$ holds. For this,
by definition of the continuous odd rearrangement we have
\begin{equation}\label{integ1}
   \int_{-L}^L  \left( \frac{d v^t}{dx } \right)^2 \! a(x) \, dx
   = \int_0^m \left\{\frac{a(\rho^t(\lambda))}{(\rho^t)'(\lambda)}
   +
   \frac{a(\rho^t(-\lambda))}{(\rho^t)'(-\lambda)}\right\}\ d\lambda.
\end{equation}

We have to compare this quantity with
$$
\int_{-L}^L \left( \frac{d v}{dx } \right)^2 \! a(x) \, dx
=\int_0^m
   \left\{\frac{a(\rho^1(\lambda))}{(\rho^1)'(\lambda)}+
   \frac{a(\rho^0(\lambda))}{(\rho^0)'(\lambda)}\right\}\ d\lambda.
$$
We will do it pointwise. Since $\sqrt{a}$ is convex the integrand in the second integral
of \eqref{integ1} is less or equal than
$$
\frac{\left(t\sqrt{a(\rho^1)} + (1-t)\sqrt{a(\rho^0)}\right)^2}{t(\rho^1)' + (1-t)(\rho^0)'}
+
\frac{\left(t\sqrt{a(\rho^0)} + (1-t)\sqrt{a(\rho^1)}\right)^2}{t(\rho^0)' + (1-t)(\rho^1)'}.
$$
Therefore, it is sufficient to prove that
\begin{equation}
\begin{split}\label{conv}
\frac{\left(t\sqrt{a(\rho^1)} + (1-t)\sqrt{a(\rho^0)}\right)^2}{t(\rho^1)' + (1-t)(\rho^0)'}
+&
\frac{\left(t\sqrt{a(\rho^0)} + (1-t)\sqrt{a(\rho^1)}\right)^2}{t(\rho^0)' + (1-t)(\rho^1)'}
\\
&\hspace{-2cm}  \leq
\frac{a(\rho^1)}{(\rho^1)'}
+
\frac{a(\rho^0)}{(\rho^0)'}.
\end{split}
\end{equation}

There are two ways to proceed now. First, let
$E:=\{\lambda \in [-m,m] : (\rho^0)'(\lambda)=+\infty\textrm{ and }
  (\rho^1)'(\lambda)=+\infty\}$. It is clear that \eqref{conv} holds
in $E$. In the set $[-m,m]\setminus E$, a	simple, but arduous, computation
shows that the previous inequality is equivalent to
\begin{equation}\label{desigualdad}
\frac{t(1-t)((\rho^0)'+(\rho^1)')}{(\rho^0)'(\rho^1)'}
\left(
(\rho^0)'\sqrt{a(\rho^1)}-(\rho^1)'\sqrt{a(\rho^0)}
\right)^2\geq0
\end{equation}
which clearly holds for all $t\in[0,1]$ (since $\rho^t$ is an increasing function
for all $t\in[0,1]$).
A second proof is the following. By symmetry one sees that
\eqref{conv} follows if we prove that
\begin{equation}\label{convpar}
\frac{\left(t\sqrt{a(\rho^1)}+(1-t)\sqrt{a(\rho^0)}\right)^2}{t(\rho^1)'+(1-t)(\rho^0)'}
\leq
t\frac{a(\rho^1)}{(\rho^1)'}+(1-t)\frac{a(\rho^0)}{(\rho^0)'}
\end{equation}
and we add this same expression by replacing $\rho^1$ by $\rho^0$ and $(\rho^0)'$ by $(\rho^1)'$.
Finally, \eqref{convpar} is easily seen to be true using that the function
$H(A,P)=A^2/P$ is a convex function in $\mathbb{R}_+\times\mathbb{R}_+$
and taking $A_i=\sqrt{a(\rho^i)}$ and $P_i=(\rho^i)'$.

It remains to prove that equality holds in  \eqref{claim1} for some $t\in (0,1)$ if and only if $v^t=v$
for all $t\in[0,1]$.
Assuming that equality holds, using any of the previous approaches
(and developing \eqref{convpar} in the second approach),
we see that all the previous inequalities become
equalities. In particular, for our $t\in(0,1)$, inequality \eqref{desigualdad}
(or \eqref{convpar}) becomes equality, which means
$$
  \frac{(\rho^1)'}{(\rho^0)'}=\sqrt{\frac{a(\rho^1)}{a(\rho^0)}}\quad
  \textrm{for every }\lambda \in [-m,m].
$$
\noindent
This is equivalent to the fact that the derivative of the function
$$
\Psi(\lambda):=\int_{-\rho(-\lambda)}^{\rho(\lambda)}\frac{ds}{\sqrt{a(s)}}
$$
vanishes in $[-m,m]$. Therefore $\Psi(\lambda)=\Psi(m)=0$ for every
$\lambda\in [-m,m]$. It follows immediately that $\rho(\lambda)=-\rho(-\lambda)$.
Hence $\rho$ is odd and its inverse $v$ too.
This automatically leads to $v^t=v$ for all $t\in[0,1]$.
\end{proof}

The following result will be useful to prove Theorem~\ref{thm:Brahms1}, that is,
that critical points of $\mathcal{E}$ are odd.
\begin{Lemma}\label{Lem:Convex_energy}
Let $a \in C^0([-L,L])$ be an even positive function such that
$\sqrt a$ is convex.
If $v \in C^1 ([-L,L])$ is increasing and $v(L)=-v(-L)=m>0$, then
\begin{equation}\label{def:ht}
h(t):=\frac{1}{2}\int_{-L}^L \left( \frac{d v^t}{dx } \right)^2 \! a(x) \, dx,\quad t\in[0,1],
\end{equation}
is a convex function.
\end{Lemma}

\begin{proof}
Note that there is a sequence of even positive functions $a_n\in C^2((-L,L))$
such that $\sqrt{a_n}$ is convex and $a_n$ tends to $a$ in $L^\infty((-L,L))$.
Consider now the function $h_n$ defined by \eqref{def:ht} with $a$ replaced by $a_n$.
If we show that $h_n$ is convex then, taking into account
\eqref{integ1} and letting $n\to\infty$, we will deduce that $h$ is convex. Therefore,
in order to prove the lemma, we can assume that $a\in C^2((-L,L))$ without loss of generality.

Let us show the existence of $a_n$ as above.
Note that $\sqrt{a}$ is differentiable a.e. since it is convex. Moreover,
$(\sqrt{a})'=a'/(2\sqrt{a})$ is nondecreasing and nonnegative in $(0,L)$, and it belongs to $L^1((0,L))$.
Using an standard convolution argument we see that there exists a sequence of
increasing positive functions $q_n\in C^\infty((-L,L))$ such that $q_n(0)=0$
and $q_n$ tends to $a'/(2\sqrt{a})$ in $L^1((0,L))$. Defining
$$
a_n(x)=\left(\sqrt{a(0)}+\int_0^x q_n(t)\,dt\right)^2\quad\textrm{for }x\in[0,L]
$$
and $a_n(x)=a_n(-x)$ for $x\in[-L,0]$ we obtain the desired sequence.

Thus, we may assume $a\in C^2((-L,L))$. Noting that
\begin{equation}\label{h(t)}
  h(t)= \frac{1}{2}\int_{-L}^L  \left( \frac{d v^t}{dx } \right)^2 \! a(x) \, dx
   = \frac{1}{2}\int_{-m}^m \frac{a(\rho^t)(\lambda)}{(\rho^t)'(\lambda)}\ d\lambda,
\end{equation}
a simple computation shows that
$$
   h'(t) =
   \frac{1}{2}\int_{-m}^{m} \left\{a'(\rho^t)\frac{\rho^1-\rho^0}{(\rho^t)'}
   -
   a (\rho^t)\frac{(\rho^1)' - (\rho^0)'}{((\rho^t)')^2}
   \right\}\, d\lambda.
$$
%

Moreover since $\sqrt{a}$ is convex, and hence $2a''/a\geq (a'/a)^2$,
we obtain
\begin{eqnarray*}
   h''(t)
   &=&
   \int_{-m}^{m} \frac{a(\rho^t)((\rho^1)' - (\rho^0)')^2}{4((\rho^t)')^3}
   \left\{
   2\frac{a''(\rho^t)}{a(\rho^t)}\frac{( \rho^1 - \rho^0 )^2}{((\rho^1)' - (\rho^0)')^2}((\rho^t)')^2\right.\\
   && \hspace{4.5cm} -
   \left.
   4  \frac{a ' (\rho^t)}{a(\rho^t)}\frac{\rho^1-\rho^0}{(\rho^1)' - (\rho^0)'}(\rho^t)'
   +
   4
   \right\}\, d\lambda\\
   &\geq&
   \int_{-m}^{m} \frac{a(\rho^t)((\rho^1)' - (\rho^0)')^2}{4((\rho^t)')^3}
   \left\{
   \frac{a ' (\rho^t)}{a(\rho^t)}\frac{\rho^1-\rho^0}{(\rho^1)' - (\rho^0)'}(\rho^t)'-2
   \right\}^2\, d\lambda
\end{eqnarray*}
which is clearly nonnegative.
\end{proof}

\begin{Remark}\label{Rmk:26}
We claim that Lemma~\ref{Lem:Dirichlet} can be deduced from Lemma~\ref{Lem:Convex_energy}.

First, making a regularization argument we can assume that $a\in C^2((-L,L))$; see the proof of
Lemma~\ref{Lem:Convex_energy}.
To prove the claim, since $h$ is a convex function such that $h'(1/2)=0$, it is clear
that $h$ is nonincreasing in $(0,1/2)$ and nondecreasing in $(1/2,1)$. In particular,
$h(1/2)\leq h(t) \leq h(0)=h(1)$ for all $t\in[0,1]$.

We want to show that either $h(t)\equiv h(1)$ or that $h(t)<h(1)$ for all $t\in(0,1)$.
Assume $h\not\equiv h(1)$ and that there exist $t_0\in(0,1)$ such that
$h(t_0)=h(1)$. Note that $h'(t_0)=0$ since $h$ is a $C^1$ function
such that $h(t) \leq h(0)=h(1)$ for all $t\in[0,1]$.
Now, since $h$ is convex and $h'(t_0)=0$, we deduce that $h(t)\geq h(t_0)=h(1)$ for all
$t\in [0,1]$. Since $h\leq h(1)$, this a contradiction with $h\not\equiv h(1)$.
\end{Remark}

Now, we prove Theorem~\ref{thm:OddRearrangement} and Remark~\ref{rem:OddRearrangement}.
\begin{proof}[Proof of Theorem~{\rm \ref{thm:OddRearrangement}}]
Thanks to the diffeomorphism $\gamma_2$ defined in \eqref{eq:ChangeVariable}, we may assume
the new weights to be $\tilde{a}$ and $\tilde{b}\equiv{1}$ being $\sqrt{\tilde{a}}$ a convex function.

(a) Let us prove that an increasing function $v\in C^1([-L,L)]$ satisfying $v(L)=-v(-L)=m>0$
and its continuous rearrangement $\{v^t\}_{0\leq t\leq 1}$ satisfy \eqref{equidistributed}.
Indeed, on the one hand
$$
\int_{-L}^L G(v)\ dx=\int_{-m}^m G(\lambda)\rho'(\lambda)\ d\lambda
=\int_0^m G(\lambda)(\rho'(\lambda)+\rho'(-\lambda))\ d\lambda.
$$
On the other hand,
\begin{eqnarray*}
\int_{-L}^L G(v^t)\ dx
&=&
\int_{-m}^m G(\lambda)(\rho^t)'(\lambda)\ d\lambda\\
&=&
t\int_{-m}^m G(\lambda)\rho'(\lambda)\ d\lambda
+
(1-t)\int_{-m}^m G(\lambda)\rho'(-\lambda)\ d\lambda\\
&=&
\int_0^m G(\lambda)(\rho'(\lambda)+\rho'(-\lambda))\ d\lambda.
\end{eqnarray*}

(b) By Lemma~\ref{Lem:Dirichlet} we only need to prove that
$t\longmapsto\mathcal{E}(v^t,(-L,L))$ is a convex function, \textit{i.e.}, part (b.3).
This follows directly from Lemma~\ref{Lem:Convex_energy} since by part (a) we have
$$
\mathcal{E} (v^{t}, (-L,L))=h(t)+\int_{-L}^L G(v^t)\ dx=h(t)+\int_{-L}^L G(v)\ dx
$$
for all $t\in[0,1]$.
\end{proof}

With the above rearrangement we can now prove Theorem~\ref{thm:Brahms1},
\textit{i.e.}, that increasing critical points
of $\mathcal{E} (\cdot, (-L,L))$ in $H^1_m((-L,L))$ are odd under assumption~\eqref{eq:BerNiren}.
As mentioned in the Introduction, this argument applies to every locally Lipschitz
$G\in C^{0,1}_{\rm loc}(\mathbb{R})$,
not only to $G\in C^{1,1}_{\rm loc}(\mathbb{R})$ as in Theorem~\ref{thm:BerestyckiNirenberg}.

\begin{proof}[Proof of Theorem~{\rm \ref{thm:Brahms1}}]
Let $u\in H^1_m((-L,L))$ be an increasing critical point of
$\mathcal{E}(\cdot,(-L,L))$. Using the change of variable $\gamma_2$ given
in~\eqref{eq:ChangeVariable}, we can work with the equivalent
problem~\eqref{eq:E-LBis} with weights $\tilde{a}=(a b) \circ \gamma_2^{-1}$
and $\tilde b \equiv 1$, whose associated functional is given by
$$
   \tilde{\mathcal{E}} (v, (-{\tilde L},{\tilde L}))
   :=\int_{-\tilde{L}}^{\tilde{L}} \left\{\frac{1}{2} v'(y)^2\tilde{a}(y)+G(v)\right\} dy .
$$
We note that $\tilde{a}$ is even and that, by~\eqref{eq:BerNiren}, we have
that $\sqrt {\tilde a}$ is convex. Furthermore
critical points of $\tilde{\mathcal{E}} (\cdot, {\tilde I})$ and $\mathcal{E} (\cdot, I)$
are related by ${\tilde u} = u \circ \gamma_2^{-1} $ and, since $\gamma_2$ is increasing,
we deduce that the critical point ${\tilde u}$ is also increasing.

We also note that $\tilde{u}\in C^1([-L,L])$. In fact, since $\tilde{u}\in
H^1_m((-L,L))$ (in particular $\tilde{u}\in L^\infty((-L,L))$) and
$G\in C^{0,1}_{\rm loc}(\mathbb{R})$ then $-(\tilde{a}\tilde{u}')' = f(\tilde{u})
\in L^\infty((-L,L))$. Therefore, $\tilde{a}\tilde{u}'\in H^1((-L,L))$, and hence,
$\tilde{u}'\in C^0([-L,L])$.

Let $\{\tilde{u}^t\}_{0\leq t\leq 1}$ be the continuous odd rearrangement
of $\tilde{u}$ and let $h$ be defined in \eqref{claim1} (with $v^t$ and $a$ replaced by $\tilde{u}^t$
and $\tilde{a}$, respectively).
By Lemma~\ref{Lem:Dirichlet} it holds either that
$h(t)<h(1)$ for all $t\in(0,1)$ or that $\tilde{u}^t=\tilde{u}$
for all $t\in[0,1]$.

Assume that $h(t)<h(1)$ for all $t\in(0,1)$. We claim that,
since $\tilde{u}$ is a solution of the associated Euler-Lagrange equation,
it holds that $h'(1)=0$. Indeed, noting that $(\tilde{u}^t)'>0$ since $u$ is
increasing, we have that $\tilde{\rho}^t$ tends to
$\tilde{\rho}=\tilde{\rho}^1$ as $t$ goes to $1$ in $C^1$.
It follows that $\tilde{u}^t$ also tends to $\tilde{u}^1=\tilde{u}$
in $C^1$ as $t$ goes to $1$.  As a consequence, since $\tilde{u}$ is a solution of the
Euler-Lagrange equation and the potential energy is constant in $t$, we deduce that $h'(1)=0$.
We now obtain a contradiction with the convexity of $h$ (given by Lemma~\ref{Lem:Convex_energy})
noting that
$$
h'(1)\geq \frac{h(1)-h(t)}{1-t} >0 \quad \textrm{for all }t\in(0,1).
$$

Therefore $\tilde{u}=\tilde{u}^t$ for all $t\in[0,1]$.
In particular, $\tilde{u}$ is an odd function, as well as
$u={\tilde u}\circ \gamma_2$ is.
\end{proof}

\begin{Remark}
In order to derive the main property~\eqref{eq:PolyaSzego}, our odd rearrangement
has been defined on the subset of {\it increasing} functions in $H^1_m (I)$. One may
wonder if there exists a more general map
$
   R: H^1_m (I) \to H^{as}_m (I),
$
where $H^{as}_m (I)$ is the subspace of $H^1_m (I)$ formed by antisymmetric functions,
satisfying
\begin{equation} \label{eq:Volare}
   R_{|H_m^{as}(I)} = {\rm Id},
   \quad\textrm{and}\quad
   \mathcal{E} (u, I) \geq \mathcal{E} (R(u) , I).
\end{equation}
However, even under the assumption~\eqref{eq:BerNiren}, it is in general
impossible to find such a map defined in the entire functional space $H^1_m (I)$.
Indeed, in Section~\ref{section5} (see Proposition~\ref{prop:NonOdd}) we
will see that, when $a \equiv b \equiv 1$, for large interval
$I$ and $m$ small enough the minimizers of $\mathcal{E} (\cdot, I)$ in $H^1_m(I)$
cannot be odd (and neither nonincreasing). Hence in this case such a map $R$ cannot exist,
since \eqref{eq:Volare}  would imply that  $R(u) $ is  an odd minimizer.
\end{Remark}

\section{Monotonicity of solutions. Proof of Theorem~\ref{thm:Increasing}}\label{sec:Monotonicity}

As stated in Theorem~\ref{thm:BerestyckiNirenberg}, Berestycki and Nirenberg established the uniqueness and
monotonicity of solutions of \eqref{E-L} under the assumptions that the weight
$a\equiv b\in C^1$ and the potential $G\in C^{1,1}_{\rm loc}$ are even functions, $a$
is log-convex, and the \textit{a priori} estimate on the solution $|u|\leq m$.
The goal of this section is to prove Theorem~\ref{thm:Increasing},
providing weaker conditions on the weights $a$ and $b$ (in particular, no convexity
assumption on them) to ensure the monotonicity
of solutions for the one-dimensional problem~\eqref{E-L}, at the price of
assuming some structural conditions on the potential $G$.

In the proof of Theorem~\ref{thm:Increasing} we use the function
\begin{equation} \label{eq:Lyapunov}
   \mathcal{H} (x, q,p) := \frac{1}{2} \big( a(x) p)^2  -  a(x) b(x) G(q) ,
\end{equation}
defined in $I \times \mathbb R^2$ (\textit{i.e.}, in the extended phase space).
Given a solution $u$ to \eqref{E-L}, we easily see that
\begin{equation} \label{eq:LaVita}
   \frac{d}{dx} \mathcal{H} (\cdot, u, u') = - (ab)' G(u) .
\end{equation}
Indeed, multiplying equation \eqref{E-L} by $au'$, we obtain
$$
  - \frac{1}{2} \big[(au')^2 \big]' + ab \big[ G(u) \big]'   =  0 ,
$$
which is equivalent to
$$
  \left\{ - \frac{1}{2}(au')^2 + ab G(u) \right\}' - (a b)' G(u)   =  0 ,
$$
and this last relation is exactly~\eqref{eq:LaVita}.
In the special case where $a$ and $b$ are constant, the function $\mathcal{H}$ is the Hamiltonian
associated to the ODE in~\eqref{E-L}.

\begin{proof}[Proof of Theorem~{\rm \ref{thm:Increasing}}] {\bf (i)}
Assume \eqref{eq:Muffin}. Let us show that any solution of \eqref{E-L} is increasing
assuming $G\geq G(-m)=G(m)$ in $\mathbb{R}$. By adding a constant to $G$, we may assume
that $G\geq G(-m)=G(m)=0$ in $\mathbb{R}$.

Given a solution $u$ of~\eqref{E-L},
let us consider in the extended phase space the associated `trajectory'
$\varphi(x) := (x, u(x), u'(x))$, $x \in (-L,L)$.
Equality~\eqref{eq:LaVita} together with the assumptions \eqref{eq:Muffin}
and $G \geq 0$ yield
\begin{equation} \label{eq:Volem}
  \frac{d}{dx} (\mathcal{H} \circ \varphi) (x) \left\{
  \begin{array}{rl}
  \geq 0   &\hbox{ for } x\in(-L,x_0],
  \vspace{1mm} \\
  \leq 0   &\hbox{ for } x\in[x_0,L).
  \end{array}
  \right.
\end{equation}
Hence $\mathcal{H} \circ \varphi$ is nondecreasing in $(-L,x_0)$ and nonincreasing in $(x_0,L)$.
It follows that
\begin{eqnarray*}
  \mathcal{H}(\varphi(x))
  &\geq&
  \min\left\{ \mathcal{H}(\varphi(-L)), \mathcal{H}(\varphi(L))\right\}
  \\
  &=&
  \frac{1}{2}
  \min\left\{ (au')^2 (-L), (au')^2 (L) \right\}
  \\
  &>& 0,
\end{eqnarray*}
where we have used $G(\pm m) = 0$, $a>0$, and $u' (\pm L) \not = 0$
(which follows from the uniqueness to the Cauchy problem and the fact that
$G'(\pm m)=0$). We conclude
$$
  \frac{1}{2}(a u')^2> a b G(u)\geq 0
  \quad \hbox{ in } [-L,L].
$$
In particular, $u'(x)> 0$ for all $x\in [-L,L]$.

{\bf (ii)} Now we prove that any solution $u$ of \eqref{E-L} is increasing
if \eqref{eq:Double:bis} holds.

By replacing $G$ by $G-G(M)$ in the equation \eqref{E-L},
we can assume without loss of generality that $G\geq 0$ in $[-M,M]$.
Since
$$
   G'(s) \leq 0  \, \textrm{ for all } s\leq -M
   \quad\textrm{and}\quad
   G'(s) \geq 0 \, \textrm{ for all } s \geq M
$$
for some constant $M$, it is easy to prove using the maximum principle that
any solution of \eqref{E-L} satisfies
\begin{equation} \label{eq:SoonOrLater}
    |u| \leq \max \{ M, m \}.
\end{equation}
In particular, since $m \geq M$, the {\it a priori} bound~\eqref{apriori}, $|u|\leq m$, holds.

Note that \eqref{eq:Volem} also holds now, since $G\geq 0$ in $\mathbb{R}$.

Moreover, by assumption \eqref{eq:Double:bis} and the maximum principle any solution
$u$ to \eqref{E-L}, with $m \geq M >0$, satisfies $-m\leq u\leq m$ (see
\eqref{eq:SoonOrLater}).

Assume by contradiction that $u$ admits a local maximum $x_1\in(-L,L)$ and a local
minimum $x_2\in(-L,L)$ satisfying $x_1<x_2$ and $\alpha_1:=u(x_1)>u(x_2)=:\alpha_2$.

First, we claim that $\alpha_1\in[-m,M]$ and $\alpha_2\in[-M,m]$. Indeed,
for instance, if $\alpha_2\in[-m,-M)$ then, by equation \eqref{E-L} and condition
\eqref{eq:Double:bis}, we get $a(x_2)u''(x_2)=b(x_2)G'(u(x_2))
\leq 0$. Therefore, using that $x_2$ is a local minimum, we have $u''(x_2)\geq 0$,
and hence, $G'(u(x_2))=-f(u(x_2))=0$. We obtain a contradiction by uniqueness of the
Cauchy problem
$$
\left\{
\begin{array}{rcl}
-(aw')'&=&bf(w),\\
w(x_2)&=&u(x_2),\\
w'(x_2)&=&0,
\end{array}
\right.
$$
by noting that $u$ and $w\equiv u(x_2)$ are two different solutions.
Thus, we have $\alpha_2\in[-M,m]$. Analogously, we obtain $\alpha_1\in[-m,M]$,
proving the claim.

Therefore, $\alpha_1\in(-M,M]$ and $\alpha_2\in[-M,M)$ (remember that
$\alpha_2<\alpha_1$).
Finally, choose $\bar{x}_1 < x_1 < x_2 < \bar{x}_2$ such that $u(\bar{x}_1)=-M$
and $u(\bar{x}_2)=M$. Since $\mathcal{H}(\varphi(\bar{x}_i))> 0$ (note that $G'(\pm M)=0$
since $G\geq G(\pm M)=0$ in $\mathbb{R}$, and thus $u'(\bar{x}_i)\neq 0$ by uniqueness)
and $\mathcal{H}(\varphi(x_i))\leq 0$ for $i=1,2$, we obtain a contradiction with
\eqref{eq:Volem}. This proves that $u$ is increasing.
\end{proof}


\begin{Remark} \label{rem:MonAme}
Let us emphasize that~\eqref{eq:BerNiren} is more restrictive than condition \eqref{eq:Muffin}
in Theorem~\ref{thm:Increasing}.
Indeed, assume that $a$ and $b$ satisfy~\eqref{eq:BerNiren}. Then there are three possible cases:
\begin{enumerate}
\item[(i)]
$(ab)'>0$ in $(-L,L)$. Then we can take $x_0=-L$ in \eqref{eq:Muffin}.
\item[(ii)]
$(ab)'<0$ in $(-L,L)$. Then we can take $x_0=L$ in \eqref{eq:Muffin}.
\item[(iii)]
There exists $x_0\in (-L,L)$ such that $(ab)'(x_0)=0$. Since
$$
  (ab)' = 2 \sqrt{ab} \, (\sqrt{ab}\, )' = 2 \sqrt{ab}\, b \, \frac{(\sqrt{ab}\, )'}{b}
$$
and $(\sqrt{ab}\, )'/b$ is nondecreasing in $(-L,L)$, it follows that $(ab)'\leq 0$
in $(-L,x_0)$ and $(ab)'\geq 0$ in $(x_0,L)$.
\end{enumerate}
As a consequence, if assumption \eqref{eq:BerNiren} holds then \eqref{eq:Muffin} also holds.
\end{Remark}


If in addition we assume $f=-G'$ to be concave in $(0,m)$ we obtain the following comparison
result between the derivatives of an increasing minimizer $u$ and its flipped $u_\star$.
We include it here even that we will not use it in the rest of the paper.
\begin{Proposition}
Assume \eqref{eq:A}, $G\in C^{1,1}_{\rm loc}(\mathbb{R})$, and $a \equiv b\in C^2$.
Let $u\in H^1_m(I)$ be an increasing minimizer of $\mathcal{E} (\cdot, I)$ and
$u_\star$ its flipped. Assume $u(0)>0$. If $f=-G'\in C^1$
is concave in $(0,m)$, then $u_\star'(x)>u'(x)$ for all $x\in(0,L]$.
\end{Proposition}

\begin{proof}
Let $u$ be an increasing minimizer and let $u_\star(x)=-u(-x)$, $x\in[-L,L]$.
Since $u(0)>0$, by minimality we see that $u_\star<u$ in $(-L,L)$ (see
Lemma~\ref{lem:OderedMin}).

Let $L_u$ be the linear operator defined by
$L_u\varphi:=\varphi''+(\log a)'\varphi'+(f'(u)+(\log a)'')\varphi$
and note that $L_u u'=L_{u_\star}u_\star'=0$ in $(-L,L)$. This can be easily obtained
differentiating \eqref{E-L} and using that $u$ and $u_\star$ are solutions of this equation.
In particular, using the assumption that $f$ is concave in $(0,m)$ and odd,
and noting that $|u_\star|<u$ in $(0,L)$ since $u_\star <u$ and $u$ is increasing, we obtain
$$
L_u(u_\star'-u')=L_u u_\star'=(f'(u)-f'(|u_\star|))u_\star' \leq 0\quad \textrm{in }(0,L).
$$
Moreover, noting that $L_u u'=0$ and $u'>0$ in $[0,L)$ we obtain that
the first Dirichlet eigenvalue $\lambda_1(L_u,(0,L))$ $>0$ (see Corollary~2.4 and Theorem~1.1 in \cite{BNV}).
Since $u_\star\leq u$ with equality at $x=L$, we deduce $(u_\star'-u')(L)\geq 0$.
Hence, since $\lambda_1(L_u,(0,L))>0$, by \cite{BNV} we can apply the maximum
principle (and later the strong maximum principle) to
$$
\left\{
\begin{array}{l}
L_u(u_\star'-u')\leq 0\quad\textrm{in }(0,L),\\
(u_\star'-u')(0)=0,\quad (u_\star'-u')(L)\geq 0,
\end{array}
\right.
$$
to obtain $u_\star'>u'$ in $(0,L)$. Finally, the fact that $u_\star'(L)>u'(L)$
easily follows by contradiction using the uniqueness for the Cauchy
problem
$$
\left\{
\begin{array}{l}
-(aw')'=bf(w),
\\
w(L)=u(L),\ w'(L)=u'(L).
\end{array}
\right.
$$
\end{proof}

\section{Uniqueness in dimension one}\label{section4}

In this section we give sufficient conditions on the weights $a$ and $b$, and
on the potential $G$, to guarantee uniqueness of solution to~\eqref{E-L}.
We start by proving the following result. When $a\equiv b$ this is exactly
Theorem~\ref{thm:Brahms2}.

\begin{Proposition} \label{prop:Salva}
Assume that~\eqref{eq:A} holds and $G\in C^{1,1}_{\rm loc}(\mathbb{R})$.
Let $L$ and $m$ be positive numbers. Suppose further that
$$
   (ab)' \geq 0 \quad \textrm{in } (0,L),
   \quad
   G \geq G(m) \quad \textrm{in } \mathbb (0,\infty),
   \quad\textrm{and}\quad
   G' \leq 0 \quad \textrm{in } (0,m).
$$
Then, problem \eqref{E-L} admits a unique solution, which is therefore odd.
Furthermore, this solution is increasing.
\end{Proposition}
\begin{proof}
The existence of minimizer, and thus of solution, is standard. Indeed,
since $a>0$ and $G\geq 0$ in $\mathbb{R}$, for a minimizing sequence we will have
$\int_I |u_{k}'|^2  \leq C$
for some constant $C$ independent of $k$. Now, let $z_k \in I$ be a zero of $u_{k}$.
We have $    |u_{m_k} (x) | \leq \left|\int_{z_k}^x |u'_{m_k}| \right|
    \leq
    (2 L)^{1/2} \left( \int_I |u_{m_k}'|^2 \right)^{1/2}
    \leq C
$
for any $x\in I$.
It follows that $u_{k}$ converges (up to a subsequence) weakly in $H^1 (I)$
and strongly in $C^0 (\overline{I})$ to some $u\in H^1_m(I)$, which will be a minimizer
(and hence a solution).

Next, let us show uniqueness of solution. Let $u$ be a solution of \eqref{E-L} and $u_\star(x)=-u(-x)$ its flipped.
By Theorem~\ref{thm:Increasing}~(i), used with $x_0=0$, we may assume that both
$u$ and $u_\star$ are increasing solutions of \eqref{E-L} and, changing $u$ by $u_\star$ if necessary, that
$u(0)\geq u_\star(0)$.

First, we claim that $u(0) = 0$. Indeed, assume by contradiction that $u(0) > 0$
and set
$$
  L_0:=\min\left\{ x\in (0,L]: u(x)=u_\star(x)\right\}.
$$
Note that $L_0 >0$ and that $u$ and $u_\star$ solve
$$
\left\{
\begin{array}{l}
  -(au')'=b f(u)\qquad \textrm{in }(-L_0,L_0),\\
  u(L_0)=-u( - L_0).
\end{array}
\right.
$$
Moreover, since $u_\star<u$ in $(0,L_0)$ and both $u$ and $u_\star$
are increasing, we have $u'(L_0)^2<u_\star'(L_0)^2$ (by uniqueness for the Cauchy problem,
or by Hopf's lemma).

Integrating~\eqref{eq:LaVita} in $(-L_0,L_0)$ we obtain
$$
\frac{a(L_0)^2}{2}\left(u'(L_0)^2-u_\star'(L_0)^2\right)
+\int_{0}^{L_0}(a b )'\Big(G(u )-G(u_\star)\Big)\ dx=0.
$$
Finally, using that $G$ is nonincreasing in $(0,m)$ we deduce that the
integrand of the previous integral is nonpositive, obtaining a contradiction.
Hence $u(0)=0$, proving the claim.


Now assume that problem~\eqref{E-L} admits two solutions $u_1$ and $u_2$ with $u_2 - u_1 \not \equiv 0$.
We know by the previous argument that $u_1(0) = u_2 (0) = 0$. Let $L_1 \in (0,L]$ be the first positive zero of
the function $u_2 - u_1$. We can assume, without loss of generality, that
$$
   (u_2 - u_1) (0) = (u_2 - u_1) (L_1) = 0
   \quad \hbox{ and }\quad
   u_2 -u_1 > 0 \hbox{ in } (0,L_1).
$$
The Hopf Lemma leads to
\begin{equation} \label{eq:Tarragona1}
  u_2'(0) > u_1'(0) > 0
  \quad \hbox{ and } \quad
  0 < u_2'(L_1) < u_1'(L_1),
\end{equation}
since we have proved that every solution is increasing.

Subtracting identity \eqref{eq:LaVita} for $u_1$ and $u_2$ and
integrating in $(0, L_1)$ we get
\begin{eqnarray}
   & &\hspace{-2cm} \frac{a(L_1)^2}{2}   \big(u_2'(L_1)^2  - u_1'(L_1)^2  \big)
   -
    \frac{a(0)^2}{2} \big( u_2'(0)^2 - u_1'(0)^2    \big)
    \nonumber \\
   &+& \int_{0}^{L_1}(a b )'\Big\{ G(u_2 )-G(u_1)  \Big\} \, dx=0 .
\label{eq:Tarragona2}
\end{eqnarray}
Using~\eqref{eq:Tarragona1} and the fact that $G$ is nonincreasing in $(0,m)$,
we reach a contradiction as before. Therefore, $u_1 \equiv u_2$.

In particular, since $u$ and $u_\star(x)=-u(-x)$ are solutions of \eqref{E-L}  we
obtain that $u=u_\star$, \textit{i.e.}, $u$ is odd.
\end{proof}

The following result was established by Berestycki and Nirenberg
in~\cite{BeresNiren}. We give here an alternative proof (which, however, also uses their sliding method).
Note that here $a$, $b$, and $G$ need not be even.

\begin{Proposition}[\cite{BeresNiren}] \label{prop:BerNir}
Assume $m>0$, $a,b\in C^1([-L,L])$ such that $a,b>0$, and $G\in C^{1,1}_{\rm loc}(\mathbb{R})$.
If $(\sqrt{ab}\,)'/b$ is nondecreasing in $(-L,L)$, then problem \eqref{E-L}
admits at most one increasing solution.
\end{Proposition}
\begin{proof}
Let $u$ be a solution of \eqref{E-L} and $\tilde{u}=u\circ\gamma_1^{-1}$,
where $\gamma_1$ is the diffeomorphism defined in \eqref{eq:ChangeVariable}.
Under this change of variables, the monotonicity of solutions is
preserved and the condition~\eqref{eq:BerNiren}, for the new weights
$\tilde{a}=\tilde{b}=\sqrt{ab}\circ\gamma_1^{-1}$, turns out to be equivalent to
the ${\rm log}$ convexity of $\tilde{a}$, \textit{i.e.},
$\tilde{a}'/\tilde{a}$ is nondecreasing in $(-L,L)$.


Thus, without loss of generality we may prove our statement for weights
$a \equiv b \in C^1 ([-L,L])$ satisfying that
$$
  \frac{a'}{a} \textrm{ is nondecreasing in }(-L,L).
$$

Let $u$ and $v$ be two increasing solutions to problem~\eqref{E-L},
and assume \textit{ab absurdo}  that
\begin{equation} \label{eq:clock}
u>v  \quad \hbox{ in } (L-\varepsilon, L)
\end{equation}
for some $\varepsilon >0$.  Consider the family of functions $(u_{\tau})_{\tau \in [0,2L)}$ defined as
$$
  u_{\tau} : I_{\tau} \to \mathbb R,
  \qquad
  x \mapsto u(x - \tau)
$$
on the interval $I_{\tau} := (-L + \tau, L)$.
Using the assumption that $a'/a$ is nondecreasing and $u' \geq 0$, we immediately see that
$$
    \frac{a'}{a} (x) u'_{\tau} (x) \geq   \frac{a'}{a} (x- \tau) u'_{\tau} (x)
    \quad
    \textrm{for all }x \in I_{\tau},
$$
and therefore
$$
  - u_{\tau}''(x) - \frac{a'}{a} (x) u'_{\tau} (x) + G'(u_{\tau})  \, \leq  \, 0
  \quad \hbox{ in } I_{\tau},
$$
\textit{i.e.}, $u_{\tau}$ is a subsolution in $I_\tau$ of the ODE in \eqref{E-L}.

Define
$$
   T:= \left\{ \tau \in [0,2L) \, \colon \, v - u_{\tau} >0 \textrm{ in } I_\tau \right\},
   \quad
   \tau_0 := \inf T ,
$$
and note that:
\begin{enumerate}
\item[(i)]
$T \not = \emptyset$. Indeed, since $u(-L) = -m$ and $v(L) =m$, we deduce
that values $\tau$ close to $2L$ belong to the set $T$. Thus, $\tau_0$ is well
defined, and by~\eqref{eq:clock} we have $\tau_0 > 0$.
\item[(ii)]
$v - u_{\tau_0} \geq 0$  in $I_{\tau_0}$ and $(v - u_{\tau_0}) (x_{\tau_0}) = 0$
for some $x_{\tau_0} \in \overline{I}_{\tau_0}$.
However, since $u$ is increasing and $\tau_0  > 0$, on the boundary of $I_{\tau_0}$ we have
$$
   u_{\tau_0} (L) = u(L - \tau_0) < m = v(L)
$$
$$
   u_{\tau_0} (-L + \tau_0) = u(-L) = - m < v(-L + \tau_0) \,.
$$
Therefore $x_{\tau_0} \not \in \partial I_{\tau_0}$.
Finally, applying the strong maximum principle on the interval $I_{\tau_0}$
(recall that $v$ is a solution and $u_{\tau_0}$ a subsolution of the nonlinear
problem), we derive a contradiction.
\end{enumerate}
\end{proof}

\begin{Remark} \label{cor:Newark}
Let $a,b \in C^1([-L,L])$ be positive functions
satisfying~\eqref{eq:BerNiren}, and $G\in C^{1,1}_{\rm loc}(\mathbb{R})$ be such
that~\eqref{eq:Double:bis} holds for some $M\in(0,m]$. Then
Theorem~\ref{thm:Increasing}~(ii) and Proposition~\ref{prop:BerNir} show that
the functional $\mathcal{E} (\cdot, (-L,L))$ admits a unique critical point in $H^1_m((-L,L))$
for any $m \geq M>0$, which is increasing (a result already stated in~\cite{BeresNiren}).
\end{Remark}

In the following corollary, $a$, $b$, and $G$ need not be even.
\begin{Corollary} \label{thm:Uniqueness}
Let $m>0$, $a,b \in C^1 ([-L,L])$ with $a,b>0$, and $G\in C^{1,1}_{\rm loc}(\mathbb{R})$. If
%
$$
    G(s) \geq G(-m)=G(m) \quad \textrm{for all } s \in \mathbb R
$$
and $\big( \sqrt{a b } \, \big)'/b$ is nondecreasing in $(-L,L)$. Then
the functional $\mathcal{E} (\cdot, (-L,L))$ admits a unique critical
point in $H^1_m((-L,L))$.
\end{Corollary}
\begin{proof}
The existence part is easily established, as in the beginning of the proof of Proposition~\ref{prop:Salva}.

Next, by Remark~\ref{rem:MonAme}, there exists $x_0\in[-L,L]$ such that \eqref{eq:Muffin}
holds. Therefore, by Theorem~\ref{thm:Increasing}~(i) any solution of \eqref{E-L}
is increasing. We conclude by applying the uniqueness result of increasing solutions
established in Proposition~\ref{prop:BerNir}.
\end{proof}

\section{Non-increasing and non-odd minimizers}\label{section5}

In this section we give conditions on the weights $a$ and $b$  for which
the minimizers of $\mathcal{E} (\cdot,(-L,L))$ in $H^1_m((-L,L))$ are either not increasing
or non-odd. Throughout this section we shall assume
\begin{equation} \label{eq:LicensePlate}
\left.
\begin{array}{c}
    a,b \in C^0(\mathbb R),
    \quad
    a, b \hbox{ even, }
    \quad
    a , b >0,
    \vspace{3mm} \\
    G \in C^0(\mathbb R),
    \quad
    G \hbox{ even, }
     \vspace{3mm} \\
    G(s) \geq G(M) = 0
    \hbox{ in } \mathbb R ,
    \quad
    G(s) > G(M) = 0
    \hbox{ in } [0,M)
   \end{array}
 \right\}
\end{equation}
for some $M >0$.


To estimate the energy value of a minimizer of $\mathcal{E}$,
we will need the following preliminary results.
\begin{Lemma}\label{Lemma3:5}
Let $a\in C^0([\alpha, \beta])$ be a positive function and $m_1,m_2
\in \mathbb{R}$. Then
$$
\min\left\{\int_\alpha^\beta a v'^2  : v\in C^1([\alpha, \beta]),
v(\alpha)=m_1, v(\beta)=m_2\right\}
=\frac{(m_2-m_1)^2}{{\int_\alpha^\beta 1/a}},
$$
and the minimum is achieved by
\begin{equation} \label{eq:Hunan}
   u(x)=(m_2-m_1)
   \frac{\int_\alpha^x 1/a }{\int_\alpha^\beta 1/a }+m_1  .
\end{equation}
\end{Lemma}
\begin{proof}
By Schwarz inequality
$$
\vert m_2-m_1 \vert =\left\vert\int_\alpha^\beta v'\right\vert \leq
\int_\alpha^\beta \sqrt{a} \vert v'\vert \frac{1}{\sqrt{a}} \leq
\left(\int_\alpha^\beta a v'^2\right)^{1/2}\left(\int_\alpha^\beta
\frac{1}{a}\right)^{1/2}.
$$
On the other hand, the minimization problem admits a unique
solution $u$ which solves
$$
\left\{
\begin{array}{l}
(a u')' = 0 \textrm{ in }(\alpha,\beta)
\\
u(\alpha) = m_1 ,\  u(\beta) = m_2 .
\end{array}
\right.
$$
We readily deduce that the solution of this Dirichlet problem is given
by~\eqref{eq:Hunan}, and a straightforward computation gives
$$
  \int_\alpha^\beta a u'^2 =
  \frac{(m_2-m_1)^2}{\int_\alpha^\beta 1/a} .
$$
\end{proof}

\begin{Proposition} \label{prop:XiMen}
Assume that \eqref{eq:LicensePlate} holds, that $m\geq 0$, and
let $ t \in [0,L)$. Then,
\begin{equation} \label{eq:XiMen}
   \inf_{v \in H^1_m ((-L,L))} \mathcal{E} (v, (-L,L))
   \, \leq \,
   \frac{M^2 + m^2 }{\int_{t}^L 1/a} + 2 \, G_1 \int_{t}^L b
\end{equation}
where $G_1 := \sup_{s \in (-m, \overline{M})} G(s) $ and
$\overline{M} = \max\{m,M\}$.
\end{Proposition}

\begin{proof}
Let $u_1$, respectively $u_2$, be the solution to the minimizing problem
$$
\min\left\{\int_{-L}^{-t} a u'^2  : u\in C^1([-L, - t]),
u(-L )= - m, \, u(-t)= M\right\},
$$
respectively,
$$
\min\left\{\int_t^L a u'^2  : u\in C^1([t, L]),
u(t)= M, \, u(L)=m\right\}.
$$
Consider the test function
$v \in H^1_m ((-L,L))$ defined by
$$
  v (x)=\left\{
  \begin{array}{ll}
  u_1
           &\mbox{ if }  -L  < x < - t ,
           \vspace{1mm}  \\
  M      &\mbox{ if }  - t \leq x \leq t,
           \vspace{1mm}  \\
  u_2   &\mbox{ if }  t  < x <  L.
\end{array}
\right.
$$
Since $G$ is even, $G(\pm M) = 0$, and the weights $a$ and $b$ are also even,  Lemma~\ref{Lemma3:5} gives
\begin{eqnarray*}
  \mathcal{E} (v, (-L,L))
  &=&
   \int_{-L}^{-t}\left\{ \frac{1}{2} a v'^2+ b G(v) \right\}+\int_t^L\left\{ \frac{1}{2} a v'^2+ b G(v) \right\}
  \\
  &=&
      \frac {(M +m)^2}{2 \int_{t}^L 1/a}+ \int_{-L}^{- t} b G(u_1)
  +  \frac {(m - M)^2}{2 \int_{t}^L 1/a} + \int_{t}^{L} b G(u_2)
  \\
  &\leq&
  \frac {M^2 + m^2}{\int_{t}^L 1/a} + 2 \sup_{s \in (-m, \overline{M})} G(s)  \int_{t}^L  b ,
\end{eqnarray*}
where we have used that both $u_1$ and $u_2$ are monotone functions, as follows from \eqref{eq:Hunan}.
\end{proof}

\subsection{Boundary perturbation of non-odd minimizers}

Recall the notation $I=(-L,L)$. We first show that the property for a minimizer
of $\mathcal{E} (\cdot, I)$ in $H^1_m(I)$ to be non-odd
is preserved under small perturbation of boundary data.

\begin{Proposition} \label{prop:CiPortera}
Assume that \eqref{eq:LicensePlate} holds.
Let $( u_{m_k})_{k=1}^{\infty}$ be a sequence of minimizers of
$\mathcal{E} (\cdot, I)$ in $H^1_{m_k}(I)$ with $0\leq m_k \rightarrow m$.
Then, up to a subsequence, we have
$$ u_{m_k} \to u_m
     \hbox{ in }  H^1(I)
$$ and $u_m$ is a minimizer of $\mathcal{E} (\cdot,I)$ in $H^1_m(I)$.
In particular $u_{m_k}\to u_m$ in $C^0(\overline{I})$.
\end{Proposition}
\begin{proof}
Since $a>0$ and $G\geq 0$ in $\mathbb{R}$, the upper bound \eqref{eq:XiMen} used
with $t=0$ gives
$$
   \int_I |u_{m_k}'|^2  \leq C
$$
for some constant $C$ independent of $m_k$.
Moreover, for each $u_{m_k}\in H^1_{m_k}(I)$, let $z_k \in \overline{I}$ be a zero of $u_{m_k}$.
The fundamental theorem of calculus yields
$$
    |u_{m_k} (x) | \leq \left|\int_{z_k}^x |u'_{m_k}| \right|
    \leq
    (2 L)^{1/2} \left( \int_I |u_{m_k}'|^2 \right)^{1/2}
    \leq C
$$
for any $x\in I$.
It follows that $u_{m_k}$ converges (up to a subsequence) weakly in $H^1 (I)$
and strongly in $C^0 (\overline{I})$ to some $u_m\in H^1_m(I)$.

Let us now prove that $u_m$ is a minimizer of $\mathcal{E} (\cdot,I)$
in $H^1_m(I)$. Indeed, take an arbitrary function $u\in H^1_m(I)$ and consider the
sequence $v_k := u + (m_k - m) \frac{x}{L}$ in $H^1_{m_k} (I)$. Note that $v_k\rightarrow u$
in $H^1(I)$ and $\mathcal{E} (u_{m_k},I)\leq \mathcal{E} (v_k,I)$. Using that
$\mathcal{E} (\cdot,I)$ is weakly lower semicontinuous we conclude that
\begin{equation}\label{*****}
\mathcal{E} (u_m,I) \leq \limsup \mathcal{E} (u_{m_k},I)
\leq \limsup \mathcal{E} (v_k,I) = \mathcal{E} (u,I),
\end{equation}
proving that $u_m$ is a minimizer.

Finally, using \eqref{*****} with $u=u_m$ we deduce that $\limsup \mathcal{E} (u_{m_k},I)
=\mathcal{E} (u_m,I)$. Therefore, since $u_{m_k}$ converges weakly to $u_m$, we have that
in fact $u_{m_k}\rightarrow u_m$ in $H^1(I)$.
\end{proof}

We can now show that, under boundary perturbation,
the property of minimizers being non-odd is preserved.
\begin{Proposition} \label{prop:LaPazzia}
Let  \eqref{eq:LicensePlate} be satisfied, $G\in C^{1,1}_{\rm loc}(\mathbb{R})$,
and  assume that for some $m_0 \geq 0$,
all minimizers of $\mathcal{E} (\cdot, I)$ in $H^1_{m_0}(I)$ are non-odd.
Then, there exists $\varepsilon >0$ such that the functional
$\mathcal{E}(\cdot, I)$ admits non-odd minimizers in $H^1_m(I)$ for each
$m \in (m_0 - \varepsilon, m_0 + \varepsilon) \cap [0, \infty)$.
\end{Proposition}
\begin{proof}
Consider a sequence of minimizer $u_{m_k}\in H^1_{m_k}(I)$ of $\mathcal{E} (\cdot,I)$ with
$m_k \to m_0$ and $m_k \geq 0$. By Proposition~\ref{prop:CiPortera}, up to a subsequence,
the sequence $u_{m_k}$ converges strongly in $C^0(\overline{I})$ to a minimizer $u_{m_0}\in H^1_{m_0}(I)$ of
$\mathcal{E} (\cdot,I)$.
Since $u_{m_0} (0) \not = 0$ ($u_{m_0} $ is not odd and recall Proposition~\ref{prop:Schubert}~(ii)),
we deduce that $u_{m_k} (0) \neq 0$ for all $m_k$ close enough to $m_0$. Thus $u_{m_k}$ is not odd.
 \end{proof}

\subsection{Non-odd minimizer}

Note that the results of the previous subsection apply with $m=0$.
This allows to give sufficient conditions on $a$ and $b$ to guarantee that
minimizers for small odd boundary data are non-odd.
\begin{Proposition} \label{prop:NonOdd}
Assume  \eqref{eq:LicensePlate}, $G\in C^{1,1}_{\rm loc}(\mathbb{R})$,
and $G(s)\leq G(0)$ for all $s\in(0,M)$. If
\begin{equation} \label{eq:KindMuckenhoupt}
     \sup_{t \in (0, L)} \left( \int_{t}^{L} \frac{1}{a}  \right) \left( \int_{0}^t b \right) > \frac{M^2}{2 G(0)},
\end{equation}
then the following holds:
\begin{enumerate}
\item[{\rm (i)}]
If $m=0$ then the minimizers of $\mathcal{E} (\cdot, I)$ in $H^1_0(I)$ are not identically zero.
\item[{\rm (ii)}]
There exists $\varepsilon >0$ such that, for each $m\in[0,\varepsilon)$, the functional
$\mathcal{E} (\cdot, I)$ admits minimizers in $H^1_m(I)$ which are non-odd and not increasing.
\end{enumerate}
\end{Proposition}

Note that (ii) applies to the unweighted case $a\equiv b\equiv 1$ whenever $I=(-L,L)$ is
large enough and $G$ satisfies \eqref{eq:LicensePlate} for some $M$.

\begin{proof}[Proof of Proposition~{\rm \ref{prop:NonOdd}}]
{\bf (i)}
Assume $m=0$. Let us show that under condition~\eqref{eq:KindMuckenhoupt} we have
\begin{equation} \label{eq:UnaTonteriaMas}
   \inf_{v \in H^1_0 (I)} \mathcal{E} (v, I) < \mathcal{E} (0, I) .
\end{equation}
Indeed, by applying Proposition~\ref{prop:XiMen} with $m=0$, we have that
inequality~\eqref{eq:UnaTonteriaMas} holds if
\begin{equation} \label{eq:DosTonteriaMas}
   \frac{M^2}{\int_{t}^L 1/a} + 2 \, G (0) \int_{t}^L b  \, <  \,  2 G(0) \int_0^L b=\mathcal{E}(0,I)
\end{equation}
for some $t \in (0,L)$.
We obtain the conclusion, by noting that inequality~\eqref{eq:DosTonteriaMas} is equivalent to
$$
 \left( \int_{t}^{L} \frac{1}{a}  \right) \left( \int_{0}^t b \right) > \frac{M^2}{2 G(0)} .
$$
{\bf (ii)}
Let $u_0\in H^1_0(I)$ be a minimizer of $\mathcal{E} (\cdot, I)$. Since
$u_0 \not \equiv 0$ by part (i), we can assume $u_0> 0$ (see Proposition~\ref{prop:Schubert}~(i)),
and in that case $u_0'(L) < 0$. In particular, $u_0$ is non-odd and not nondecreasing. We conclude by applying Propositions~\ref{prop:CiPortera} and \ref{prop:LaPazzia} (see Figure~\ref{fig2}).
\end{proof}

\begin{Example}
Condition \eqref{eq:KindMuckenhoupt} holds true for $L$ large enough
for any positive and even function $a: \mathbb R \to \mathbb R$ satisfying, for instance,
$$
    \int_0^{+\infty} \frac{1}{a} = +\infty,
$$
independently of the weight $b$. In particular, it holds for $a \equiv 1$.

In this case, by Proposition~\ref{prop:NonOdd}, there exist $L_0>0$ and $\varepsilon > 0$ such that
the minimizers of $\mathcal{E}$ in $H^1_m((-L,L))$ are non-odd and not increasing
whenever $ L > L_0$ and $m \in [0, \varepsilon)$.
\end{Example}

The above result gives a class of weights for which minimizers are not
odd in large intervals and small boundary values.
To obtain similar results for ``large" boundary data (such as $u(\pm L)=\pm M$;
recall that we are assuming \eqref{eq:LicensePlate} and that $M=1$ in the
Allen-Cahn nonlinearity), we will look for conditions on the weights $a$ and $b$
to ensure
\begin{equation} \label{eq:Hubei}
   \inf_{u \in H^{as}_m (I)} \mathcal{E} (u, I)
   \, > \,
   \inf_{u \in H^{1}_m (I)} \mathcal{E} (u, I) .
\end{equation}
Recall that $H^{as}_m(I)$ is formed by those functions in $H^1_m(I)$ which are odd.

Note that Proposition~\ref{prop:XiMen} gives an upper bound for the
right hand-side value in~\eqref{eq:Hubei}.
The following proposition gives now a lower bound on $\min_{u \in H^{as}_m (I)} \mathcal{E} (u, I)$.

\begin{Proposition}\label{acotadas}
Assume that \eqref{eq:LicensePlate} holds and $m>0$.
Then, there exists a positive constant $C^{as}$ depending only on
$a$, $b$, $G$, $M$, and $m$ $($but independent of $L$$)$
such that
\begin{equation}  \label{eq:BellaCiao}
 \mathcal{E} (u, I) \geq C^{as} >0
 \quad \textrm{for all } u \in H^{as}_m (I), \text{ where } I=(-L,L) .
\end{equation}
Moreover, the constant $C^{as}$ can be chosen as
\begin{equation}\label{Cas}
   C^{as} := \inf_{t >0}
   \Big\{ \frac{m_0^2}{\int_0^t 1/a}
  +
  2 G_0 \int_0^t b  \Big\} ,
\end{equation}
where
\begin{equation}\label{m0_and_G0}
    m_0 :=  \frac{1}{2}\min\{m,M\}
    \quad \hbox{ and } \quad
    G_0 := \inf_{s \in (0, m_0)} G(s) .
\end{equation}
\end{Proposition}
\begin{proof}
Let $u \in H^{as}_m (I)$ be such that
$$
  \mathcal{E} (u,I) = \inf_{v \in H^{as}_m (I)} \mathcal{E} (v, I)\,.
$$

Since $u(0) =0$, we can choose $\beta \in (0,L)$ be such that
$$
  |u(\beta)|=  m_0
  \quad\textrm{and}\quad
  |u(x)| <  m_0
  \quad \textrm{for all } x \in (0, \beta)  \,.
$$
Setting $G_0 := \inf_{s \in (0, m_0)}  G(s)$ we note that
$G(u(x))\geq G_0>0$ for all $x\in[0,\beta]$. This inequality together with
Lemma~\ref{Lemma3:5}
(with boundary conditions $u(0) =0$, $u(\beta) = m_0$ or $u(\beta) = -m_0$) yield
\begin{eqnarray}
  \mathcal{E} (u, I)
  &\geq&
  \int_{-\beta}^\beta \Big\{ \frac{1}{2} a u'^2+  b G(u) \Big\}
  =
  \int_0^\beta \Big\{a u'^2+  2 b G(u) \Big\}
  \nonumber \\
  &\geq&
  \frac {m_0^2}{\int_0^\beta 1/a}
  + 2 G_0 \int_0^\beta b.
  \label{key}
\end{eqnarray}

Define this last expression as a function of $t>0$, namely
$$
  \Psi(t):=
  \frac {m_0^2}{\int_0^t 1/a}
  +
  2 G_0 \int_0^t b.
$$
Clearly $\lim_{t\to 0}\Psi(t)=+\infty$ and, since $G_0>0$ by the last
assumption in \eqref{eq:LicensePlate} and the fact that $m_0<M$, $\lim_{t \to
+\infty}\Psi(t)\in(0,+\infty]$. Since $\Psi$ is positive and
continuous in $(0,\infty)$, it holds that $C^{as}=\inf_{t>0}\Psi(t)>0$
and $C^{as}$ depends only on $a$, $b$, $m_0$, and $G_0$. This combined with
\eqref{key} proves the result.
\end{proof}

Propositions~\ref{prop:XiMen} and \ref{acotadas} yield immediately the following result.
\begin{Corollary}
Assume that \eqref{eq:LicensePlate} holds, $m>0$, and that
$$
   \inf_{t \in (0,L)}
   \left\{ \frac{M^2 + m^2 }{\int_{t}^L 1/a} + 2 \, G_1 \int_{t}^L b \right\}
   \, < \, C^{as},
$$
where $G_1 := \sup_{s \in (-m, \overline{M})} G(s)$,
$\overline{M} = \max\{m,M\}$, and $C^{as}$ is defined by
\eqref{Cas}-\eqref{m0_and_G0}.
Then
$$
\inf_{u \in H^{as}_m (I)} \mathcal{E} (u, I)   >   \inf_{u \in H^{1}_m (I)} \mathcal{E} (u, I).
$$
In particular, the minimizers of $\mathcal{E} (\cdot, I)$ in $H^1_m(I)$ are not odd.
\end{Corollary}


Next, by setting
$$
   \Phi_m (L) := \inf_{u \in H^{1}_m (I)} \mathcal{E} (u, I),\quad L>0\,,
$$
we characterize the weights $a$ and $b$ for which
$\liminf_{L\rightarrow+\infty}\Phi_m (L)=0$. This, jointly with Proposition~\ref{acotadas},
will provide a first class of weights that guarantee~\eqref{eq:Hubei}. That
is, a class of weights for which the minimizers of $\mathcal{E}(\cdot,I)$ in
$H^1_m(I)$ are not odd.



\begin{Proposition}\label{infimo=0}
Assume that \eqref{eq:LicensePlate} holds and $m>0$. The following assertions are
equivalent:
\begin{enumerate}
\item[{\rm (i)}]
$\liminf_{L\rightarrow+\infty}\Phi_m (L)=0$;
\item[{\rm (ii)}]
There exists a sequence of bounded intervals $J_n=[\alpha_n,\beta_n] \subset \mathbb{R}$ $(\alpha_n < \beta_n)$
satisfying
\begin{equation} \label{eq:Salva}
   \int_{J_n}\frac{1}{a}\rightarrow +\infty
   \quad \hbox{ and } \quad
   \int_{J_n} b\rightarrow 0.
\end{equation}
\end{enumerate}
\end{Proposition}
\begin{proof}
(i) $\Rightarrow$ (ii)
Set $I_n := (-L_n, L_n)$ with $L_n\to\infty$. Let $u_n\in H^1_m(I_n)$ be a minimizer
of $\mathcal{E} (\cdot, I_n)$ and assume that $\Phi_m (L_n)= \mathcal{E} (u_n, I_n)\to 0$.

Let $G_0$ and $m_0$ be given in \eqref{m0_and_G0}.
Consider an interval $J_n :=[\alpha_n, \beta_n] \subset I_n$ such that
$$
  u_n ( \alpha_n)=0,\quad u_n ( \beta_n)=m_0,
  \quad \hbox{ and } \quad
  |u_n(x)| \leq m_0
  \quad \textrm{for all }x \in (\alpha_n, \beta_n).
$$
Since $G(u_n(x)) \geq G_0 >0$ for all $x \in (\alpha_n, \beta_n)$ by \eqref{eq:LicensePlate},
applying Lemma~\ref{Lemma3:5} on the interval $ (\alpha_n, \beta_n)$, we conclude
$$
   \mathcal{E} (u_n, I_n)
   \geq
   \int_{J_n} \left\{ \frac{1}{2} a u_n'^2+ b G(u_n)\right\}
   \geq
   \frac {m_0^2}{2\int_{J_n} 1/a}+ G_0 \int_{J_n} b.
$$
This last inequality proves the assertion.


(ii) $\Rightarrow$ (i)
We claim that we can assume, without loss of generality, that the sequence of intervals
$J_n=[\alpha_{n}, \beta_{n}] \subset [0,\infty)$ instead of $J_n\subset\mathbb{R}$.
Indeed, note that we can suppose that $\beta_n>0$ changing $[\alpha_n,\beta_n]$ by
$[-\beta_n,-\alpha_n]$ if necessary. Moreover, if $\alpha_n\leq0$ then
$$
0\leq \int_{\alpha_n}^0 b \leq \int_{J_n} b \rightarrow 0
\quad \textrm{and}\quad
0\leq \int_0^{\beta_n} b \leq \int_{J_n} b \rightarrow 0
$$
by \eqref{eq:Salva}. This proves that $\alpha_n$ and $\beta_n$ tend to zero as $n$
goes to infinity, a contradiction with \eqref{eq:Salva}:
\begin{equation}\label{new:a}
\int_{J_n}\frac{1}{a}=\int_{\alpha_n}^{\beta_n}\frac{1}{a}\rightarrow +\infty.
\end{equation}
Therefore, we can assume, up to a subsequence, that $\alpha_n\geq 0$ proving the claim.

Let $J_n=[\alpha_{n}, \beta_{n}] \subset [0,\infty)$ be a sequence of bounded intervals
satisfying~\eqref{eq:Salva}.
Inequality~\eqref{eq:XiMen} applied with $t= \alpha_n$ and $L=\beta_n$ gives
$$
   0 \leq
   \Phi_m (\beta_n)
   \, \leq \,
  \frac {M^2+m^2}{\int_{J_n} 1/a}+  2G_1\int_{J_n} b,
$$
where $G_1=\sup_{s\in(-m,\overline{M})} G(s)$ and $\overline{M} = \max\{m,M\}$.
We conclude the proof noting that
the right hand-side of the previous inequality tends to zero by \eqref{eq:Salva}
and that $\beta_n\rightarrow+\infty$ by \eqref{new:a}.
\end{proof}

Now, as stated in Proposition \ref{cor:characteriation}, we are
able to exhibit a class of weights $a$, $b$, for which
the minimizers of  $\mathcal{E} (\cdot, I)$ in $H^1_m(I)$ are not odd when the domain is large.

\begin{proof}[Proof of Proposition {\rm\ref{cor:characteriation}}]
By the hypothesis of the proposition, \eqref{eq:LicensePlate} is satisfied taking $M:=m$,
after replacing $G$ by $G-G(M)$. Now, on the one hand, by Proposition \ref{acotadas}
there exists a constant $C^{as}>0$
(independent of the interval $I$) such that $\mathcal{E} (u, I)\geq
C^{as}$ for all $u \in H_m^{as}(I)$. On the other hand, note that
$\Phi_m$ is a nonincreasing function, \textit{i.e.}, $\Phi_m(L_2)\leq \Phi_m(L_1)$ for all $L_1<L_2$.
This follows by noting that given $u\in H^1_m((-L_1,L_1))$ we can extend it
to $\tilde{u}\in H^1_m((-L_2,L_2))$:
$$
\tilde{u}:=
\left\{
\begin{array}{ccc}
-m&\textrm{in}&(-L_2,-L_1),\\
u&\textrm{in}&(-L_1,L_1),\\
m&\textrm{in}&(L_1,L_2),
\end{array}
\right.
$$
and $\mathcal{E}(u,(-L_1,L_1))=\mathcal{E}(\tilde{u},(-L_2,L_2))$.
Therefore, by Proposition~\ref{infimo=0} we can take $L_0>0$ such that
$\Phi_m(L)<C^{as}$ for all $L\geq L_0$. As a consequence, minimizers of
$\mathcal{E} (\cdot,I)$ in $H^1_m(I)$ cannot be odd for $L\geq L_0$.
\end{proof}


\begin{Remark}
Note that if $ab$ is increasing and positive in $[0,\infty)$, then condition~\eqref{eq:Salva}
cannot hold and therefore we cannot use Proposition~\ref{cor:characteriation} to obtain non-oddness
of minimizers on large intervals. Indeed, the previous assertion on condition ~\eqref{eq:Salva}
follows from
$$
   \int_{J_n} b \, = \, \int_{J_n} \frac{ab}{a}
                \, \geq \, (ab) (0) \int_{J_n} \frac{1}{a}.
$$
This is consistent with the result of Theorem~\ref{thm:Brahms2} where we
proved that problem \eqref{E-L} admits a unique solution, which is therefore
odd, under the assumption that $a\equiv b$ is nondecreasing in $(0,L)$.

\end{Remark}


Finally, we point out that we could give a precise quantitative result for the minimum length of the interval $L$
(in terms of lower and upper bounds on $a$ and $b$, and of the nonlinearity $G$)
guaranteeing non-oddness of minimizers.

\section{Uniqueness results in higher dimensions} \label{sec:ndim}
In this section we consider a bounded domain $\Omega\subset\mathbb{R}^N$,
a reflection with respect to a hyperplane, $\sigma:\mathbb R^N \to \mathbb R^N$, that leaves $\Omega$ invariant, and
\begin{equation}\label{6.0}
  A \in C^2 (\overline{\Omega}, S_N(\mathbb R) ),
  \, \, \,
  0 < b \in C^1 (\overline{\Omega}, \mathbb R),
  \, \, \,
  \varphi\in (H^1\cap L^\infty)(\Omega),
\end{equation}
where $S_N (\mathbb R)$ stands for the set of $N \times N$ symmetric matrices with real coefficients. We
will assume that
\begin{equation} \label{eq:UnifCoercivity}
    \langle A(x) \xi, \xi \rangle \geq c_0 |\xi|^2
    \quad
   \textrm{for all }x\in\Omega\textrm{ and } \xi \in \mathbb R^N,
\end{equation}
for some positive constant $c_0$. The potential $G\in C^2 (\mathbb R)$ will satisfy
\begin{equation}\label{newG:bis}
\begin{array}{l}
\textrm{there exists }M>0\textrm{ such that } G'(s)\leq 0\textrm{ for all }s< -M
\\
\textrm{and }G'(s)\geq 0\textrm{ for all }s>M
\end{array}
\end{equation}
for some constant $M>0$.
When discussing the antisymmetry property of the solution, we shall also assume
\begin{equation} \label{eq:Flughafen:bis}
  G(s) = G(-s)\textrm{ in }\mathbb{R},
  \quad
  A \circ \sigma = A,
  \quad
  b \circ \sigma = b ,
  \quad
  \varphi \circ \sigma =  - \varphi.
\end{equation}

With these assumptions, we address the question of uniqueness of critical points for the functional
\begin{equation}\label{eq:NDimFunctional:bis}
   \mathcal{E} (u, \Omega)
   :=
   \int_{\Omega}
   \left\{  \frac{1}{2} \langle A(x) \nabla u,  \nabla u \rangle  + b(x) G(u) \right\} dx,\quad
   u\in H^1_\varphi(\Omega),
\end{equation}
and also of their antisymmetry property. We work with the functional spaces
$$
  H^{1}_{\varphi} (\Omega)
  := \{ u \in H^1 (\Omega) \, \colon \, u - \varphi \in H^1_0 (\Omega) \}
$$
and
$$
  H^{as}_{\varphi} (\Omega)
  := \{ u \in H^1_{\varphi} (\Omega) \, \colon \, u \circ \sigma = - u\} .
$$

Let us first emphasize that critical points $u$ are weak solutions to the
Euler-Lagrange equation
\begin{equation}\label{new116}
   - \hbox{div} (A(x) \nabla u)  + b(x) G'(u) = 0,
   \quad
   u \in H^1_{\varphi} (\Omega).
\end{equation}
As a consequence, if \eqref{newG:bis} holds then
\begin{equation}
\label{linf}
 |u|\leq \max\{M,\|\varphi\|_{L^\infty(\partial\Omega)}\} \qquad \text{ in } \Omega
\end{equation}
by \eqref{newG:bis} and the maximum principle. Therefore, since $u\in L^\infty(\Omega)$
then $G'(u)\in L^\infty(\Omega)$, and hence problem \eqref{new116}
can be understood in the distributional sense.

The following existence result states, in particular, that $\mathcal{E}$ always admits
an antisymmetric critical point under assumption \eqref{eq:Flughafen:bis}.
\begin{Proposition} \label{prop:ExistenceMin}
Assume \eqref{6.0}, \eqref{eq:UnifCoercivity}, and \eqref{newG:bis}.
Both functionals, $\mathcal{E}$ and when assuming \eqref{eq:Flughafen:bis} its restriction
$\mathcal{E}|_{H^{as}_{\varphi} (\Omega)}$, admit a minimizer. Moreover, both minimizers are
critical points of $\mathcal{E}$ in $H^1_{\varphi}(\Omega)$.
\end{Proposition}
\begin{proof}
The fact that $\mathcal{E}$ and the restriction $\mathcal{E}|_{H^{as}_{\varphi} (\Omega)}$ admit
a minimizer $u_0 \in H^{1}_{\varphi} (\Omega)$ and $u^{as} \in H^{as}_{\varphi} (\Omega)$,
respectively, follows by applying standard results of the calculus of variations
(note that $G$ is bounded from below by \eqref{newG:bis}
and that $\mathcal{E}(\varphi,\Omega)<+\infty$).

To show that $u^{as}\in H^{as}_{\varphi} (\Omega)$ is a critical point of $\mathcal{E} (\cdot, \Omega)$,
write any $\xi \in C_0^\infty(\Omega)$ as $\xi = \xi^{s} + \xi^{as}$ with
$\xi^{s}$ and $\xi^{as}$ to be symmetric and antisymmetric respectively.
Obviously $D \mathcal{E} (u^{as})\xi^{as} = 0$ in the weak sense. Furthermore, due
to the symmetry assumptions on $A$, $b$, and $G$, we readily see that the functions
$$
  x \mapsto  \langle A( x ) \nabla u^{as} (x),   \nabla \xi^{s} (x)  \rangle
  \qquad
  x \mapsto  b (x) G'(u^{as} (x) ) \xi^{s} (x)
$$
are antisymmetric. Therefore $D \mathcal{E} (u^{as})\xi^{s} = 0$.
We conclude $D \mathcal{E} (u^{as})\xi = 0$ for any
$\xi \in C_0^{\infty} (\Omega)$.
\end{proof}

Our next proposition establishes uniqueness of critical points of $\mathcal{E}$ when the second
variation of $\mathcal{E}$ at $u\equiv0$ is nonnegative and, in addition,
$-G''(0) > -G''(s)$ for all $s\neq0$. Let us note that this last condition
on $G$ is satisfied by the double well potential $G(s) = (1-s^2)^2/4$.

Before stating our result, let us recall that the second variation of energy at $u\equiv 0$
is nonnegative whenever
$$
  D^2\mathcal{E}(0)(\xi,\xi)
  :=\int_{\Omega} \{ \langle A(x) \nabla \xi , \nabla \xi \rangle + b(x) G''(0) \xi^2 \}\,dx
  \, \geq \, 0
  \quad
  \textrm{for all }\xi \in H^1_0 (\Omega).
$$
If in addition $u\equiv0$ is a solution of \eqref{new116}, then we say that $u\equiv0$ is a
semi-stable solution. Instead, if $u\equiv0$ is a solution of \eqref{new116} and
$D^2\mathcal{E}(0)$ is not nonnegative definite, then we say that $u\equiv0$ is an unstable
solution.

\begin{Remark}\label{RemD^2}
By considering the eigenvalue
$$
  \lambda_1 (A,b, \Omega) :=
  \inf\left\{ \frac{\int_{\Omega} \langle A(x) \nabla \xi, \nabla \xi \rangle\,dx}
                   {\int_{\Omega} b(x) \xi^2\,dx}
       \, \colon \, \xi \in H^1_{0} (\Omega), \, \xi \not \equiv 0
  \right\} ,
$$
we easily see that $D^2\mathcal{E}(0) \geq 0$ if and only if $\lambda_1(A,b, \Omega) \geq - G''(0)$.
\end{Remark}

We establish the following antisymmetry result and a kind of converse to it.

\begin{Proposition}\label{Proposition:Marc1}
Assume \eqref{6.0}, \eqref{eq:UnifCoercivity}, and \eqref{newG:bis}.
The following assertions hold:
\begin{enumerate}
\item[{\rm (i)}]
Assume that $D^2\mathcal{E}(0) \geq 0$ and that $-G''(0) > -G''(s)$ for all $s\neq 0$.
Then, for every $\varphi\in (H^1\cap L^\infty)(\Omega)$, $\mathcal{E} (\cdot, \Omega)$
admits a unique critical point in $H^{1}_\varphi (\Omega)$.
In addition, under condition~\eqref{eq:Flughafen:bis} it is antisymmetric.
\item[{\rm (ii)}]
Assume \eqref{eq:Flughafen:bis} and that $D^2\mathcal{E}(0)(\xi,\xi) < 0$ for some
$\xi\in H^1_0(\Omega)$. Then, for some boundary values $\varphi\in C^1(\overline{\Omega})$, the minimizers
of $\mathcal{E} (\cdot, \Omega)$ in $H^{1}_\varphi(\Omega)$ are not antisymmetric.
\end{enumerate}
\end{Proposition}
\begin{proof}
{\bf (i)}
Let $u_1, u_2 \in H^1_{\varphi} (\Omega)$ be two critical points of $\mathcal{E} (\cdot, \Omega)$
and recall that $u_1,u_2\in L^\infty(\Omega)$ by \eqref{linf}.
Define $u^t:=tu_1+(1-t)u_2$, $t\in[0,1]$.
Since $G''(s)\geq G''(0)$ for all $s$, we have
\begin{eqnarray}\label{second:der}
\frac{d^2}{dt^2}\mathcal{E}(u^t,\Omega)
&=&
\int_\Omega \left\{\langle A(x)\nabla (u_1-u_2),\nabla (u_1-u_2)\rangle+b(x)G''(u^t)(u_1-u_2)^2\right\}\,dx
\nonumber\\
&\geq&
\int_\Omega \left\{\langle A(x)\nabla (u_1-u_2),\nabla (u_1-u_2)\rangle+b(x)G''(0)(u_1-u_2)^2\right\}\,dx
\nonumber\\
&\geq&
0,
\end{eqnarray}
where in the last inequality we have used that $u_1-u_2\in H^1_0(\Omega)$ and that $D^2\mathcal{E}(0) \geq 0$.
Therefore, $h(t):=\mathcal{E}(u^t,\Omega)$ is a convex function. Moreover,
since $u_1$ and $u_2$ are critical points, we have that $h'(0)=h'(1)=0$. It follows that
$h$ is constant in $[0,1]$.

As a consequence, the left hand side of \eqref{second:der} is zero
and, thus, all the inequalities in \eqref{second:der}
become equalities. Hence, since $-G''(0) > -G''(s)$ for all $s\neq 0$, we obtain that
$u^t=tu_1+(1-t)u_2=0$ (\textit{i.e.}, $u_1=t^{-1}(t-1)u_2$) in $\{x\in\Omega: u_1(x)\neq u_2(x)\}$.
Since this must hold for all $t\in(0,1)$, we deduce that $\{x\in\Omega: u_1(x)\neq u_2(x)\}=\varnothing$.
Thus, $\mathcal{E} (\cdot, \Omega)$ admits a unique critical point $u$.

Under the additional condition~\eqref{eq:Flughafen:bis}, we have that
$- u \circ \sigma$ is also a critical point. By
uniqueness we must have $- u \circ \sigma = u$. Thus, $u$ is antisymmetric.

{\bf (ii)}
Assume now that $D^2\mathcal{E}(0)(\xi,\xi) < 0$ for some
$\xi\in H^1_0(\Omega)$.
Let $\varphi_n\in C^1(\overline{\Omega})$ be any sequence converging to zero in
$C^1(\overline{\Omega})$ and let $u_n\in H^1(\Omega)$ be a minimizer of $\mathcal{E}(\cdot,\Omega)$
in $H^1_{\varphi_n}(\Omega)$. Let us prove that, for $n$ large enough, $u_n$ is not antisymmetric.

We claim that $u_n\rightharpoonup u_0$ in $H^1(\Omega)$ (up to a subsequence) and that $u_0$ is a minimizer of
$\mathcal{E}(\cdot,\Omega)$ in $H^1_0(\Omega)$.
Indeed, since $u_n$ is a minimizer we have $\mathcal{E}(u_n,\Omega)\leq \mathcal{E}(\varphi_n,\Omega)\leq C$
for some constant $C$ independent of~$n$.
In particular, $\int_\Omega|\nabla u_n|^2\,dx\leq C$ and therefore
$$
\int_\Omega|\nabla (u_n-\varphi_n)|^2\, dx\leq C.
$$
As a consequence, since $(u_n-\varphi_n)|_{\partial\Omega}\equiv 0$ and thus
$u_n-\varphi_n\in H^1_0(\Omega)$, a subsequence of
$u_n-\varphi_n$ converges weakly. Thus, up to a subsequence, $u_n\rightharpoonup u_0$ in
$H^1(\Omega)$ for some $u_0\in H^1(\Omega)$.

Let $u\in H^1_0(\Omega)$. Note that $u+\varphi_n\in H^1_{\varphi_n}(\Omega)$.
By minimality of $u_n$ we have $\mathcal{E}(u_n,\Omega)\leq \mathcal{E}(u+\varphi_n,\Omega)$.
Using the semicontinuity of the $H^1$-norm and Fatou's lemma, we obtain (taking the $\liminf$) that
$\mathcal{E}(u_0,\Omega)\leq \mathcal{E}(u,\Omega)$. That is, $u_0$ is a minimizer in $H^1_0(\Omega)$.

Finally, assume by contradiction that, for $n$ large enough, $u_n$ is antisymmetric with respect to a hyperplane. Then
\begin{equation}\label{lalalala}
0=\int_\Omega u_n\,dx\longrightarrow \int_\Omega u_0\, dx\quad\textrm{as }n\rightarrow+\infty.
\end{equation}
Since $D^2\mathcal{E}(0)(\xi,\xi) < 0$ for some $\xi\in H^1_0(\Omega)$, any minimizer $u_0$ of
$\mathcal{E} (\cdot, \Omega)$
in $H^1_0(\Omega)$ cannot be identically zero (otherwise $u\equiv 0$ would be a minimizer,
and would be unstable by hypothesis, a contradiction).
In addition, $u_0$ has constant sign in $\Omega$ (since its absolute value also
minimizes; one uses here an argument as in Proposition~\ref{prop:Schubert}).
This and \eqref{lalalala} give a contradiction, obtaining that for $n$ large enough
$u_n$ is a minimizer of $\mathcal{E}(\cdot,\Omega)$ in $H^1_{\varphi_n}(\Omega)$ which is not antisymmetric.
\end{proof}

As we said before, $D^2\mathcal{E}(0)\geq 0$ is equivalent to $\lambda_1(A,b,\Omega)\geq -G''(0)$
(see Remark~\ref{RemD^2}).
Therefore, the first part of the previous proposition can be reformulated as follows.

\begin{Corollary}\label{Cor:unique_crt}
Assume \eqref{6.0}, \eqref{eq:UnifCoercivity}, and \eqref{newG:bis}.
If $\lambda_1 (A,b,\Omega) \geq - G''(0) > - G''(s)$ for all $s\neq 0$,
then $\mathcal{E} (\cdot, \Omega)$ admits a unique critical point in
$H^1_{\varphi}(\Omega)$ for every $\varphi\in (H^1\cap L^\infty)(\Omega)$. In addition,
under assumption~\eqref{eq:Flughafen:bis} it is antisymmetric.
\end{Corollary}

The following result gives a lower bound for the eigenvalue
$\lambda_1(A,b,\Omega)$ and will be useful in order to apply
Corollary~\ref{Cor:unique_crt}.
\begin{Proposition} \label{prop:Elizabeth}
Assume \eqref{6.0} and that $A(x)$ is a nonnegative symmetric matrix for all $x\in\overline\Omega$.
The following inequality holds:
\begin{equation}\label{eq:Anna}
  \lambda_1 (A,b ,\Omega)
  \, \geq \,
  \frac{1}{4}
  \inf_{\Omega} \left\{ 2 \, {\rm div} \left( A (x) \frac{\nabla b}{b^2}  \right)
  +
  \Big\langle A (x) \nabla b , \frac{\nabla b}{b^3} \Big\rangle
  \right\} .
\end{equation}
In particular, if $A \equiv a\, {\rm Id}$ and $a  \equiv b$, then
\begin{equation} \label{eq:AnnaRoja}
   \lambda_1 (A,b ,\Omega) \geq \inf_{\Omega}\frac{\Delta \sqrt a}{\sqrt a} .
\end{equation}
\end{Proposition}
\begin{proof}
Noting that the map
$$
   \xi \longmapsto \eta:=b^{1/2} \xi ,
$$
is a bijection from $H^1_0(\Omega)$ to itself, we obviously have
$$
   \lambda_1 (A,b,\Omega)
   =
   \inf_{\eta \in H^1_0 (\Omega) \setminus \{ 0 \}}
     \frac{\int_{\Omega}  \big\langle A (x) \nabla (\eta b^{-1/2}), \nabla (\eta b^{-1/2}) \big\rangle\,dx }
     {\int_{\Omega} \eta^2\,dx}.
$$
For $\eta \in H^1_0(\Omega) \setminus \{ 0 \}$, using the fact that
$\langle A\nabla\eta,\nabla\eta\rangle\geq 0$, we get
\begin{eqnarray}
  \big\langle A (x) \nabla (\eta b^{-1/2}), \nabla (\eta b^{-1/2}) \big\rangle
  &=&
  \frac{1}{b}
  \Big\langle A (x) \big( \nabla \eta - \frac{\eta}{2} \frac{\nabla b}{b} \big) ,
  \nabla \eta - \frac{\eta}{2} \frac{\nabla b}{b}
  \Big\rangle
  \nonumber\\
  &\geq&
  - \eta
  \Big\langle A (x) \frac{\nabla b}{b^2} , \nabla \eta \Big\rangle
  +
  \frac{\eta^2}{4 b}
  \Big\langle A (x) \frac{\nabla b}{b}  , \frac{\nabla b}{b} \Big\rangle
  \nonumber \\
  &=&
   -
  \Big\langle  A(x) \frac{\nabla b}{b^2} , \nabla \big(\frac{\eta^2}{2} \big) \Big\rangle
  +
  \frac{\eta^2}{4 }
  \Big\langle A (x) \nabla b , \frac{\nabla b}{b^3} \Big\rangle .
  \nonumber
\end{eqnarray}
Integrating and applying the divergence theorem, we get

\begin{eqnarray*}
& & \hspace{-2cm} \int_{\Omega} \big\langle A (x)  \nabla(\eta b^{-1/2}), \nabla(\eta b^{-1/2}) \big\rangle\,dx
\\ & \geq  &
  \frac{1}{4}
  \int_{\Omega}
  \left\{ 2 \hbox{div}  \left( A (x) \frac{\nabla b}{b^2} \right)
  +
  \Big\langle A (x) \nabla b , \frac{\nabla b}{b^3} \Big\rangle
  \right\} \eta^2\,dx
\end{eqnarray*}
and therefore
$$
  \lambda_1 (A,b, \Omega)
  \, \geq \,
  \frac{1}{4}
  \inf_{\Omega} \left\{ 2 \hbox{div}  \left( A (x)  \frac{\nabla b}{b^2}  \right)
  +
  \Big\langle A (x) \nabla b , \frac{\nabla b}{b^3} \Big\rangle
  \right\} ,
$$
which is inequality~\eqref{eq:Anna}.

In the case $A \equiv a\, {\rm Id}$ and $a \equiv b$, inequality~\eqref{eq:Anna} turns out to be equivalent to
$$
  \lambda_1 (a\, {\rm Id},a,\Omega)
  \geq
    \frac{1}{4}\inf_{\Omega} \left\{ 2 \Delta (\log a ) + |\nabla (\log a)|^2\right\}
  =
    \inf_{\Omega}\frac{\Delta\sqrt{a}}{\sqrt{a}}.
$$
\end{proof}

Theorem~\ref{thm:ndim} and Corollary~\ref{cor:OddHigher} follow as an immediate consequence
of Corollary~\ref{Cor:unique_crt} and Proposition~\ref{prop:Elizabeth}.

\begin{proof}[Proof of Theorem~{\rm \ref{thm:ndim}} and Corollary~{\rm \ref{cor:OddHigher}}]
Assume $A \equiv a\, {\rm Id}$ and $a  \equiv b$. By Proposition~\ref{prop:Elizabeth} and
the assumptions of the theorem, we have
$$
\lambda_1 (A,b ,\Omega) \geq \inf_\Omega\frac{\Delta \sqrt a}{\sqrt a}\geq -G''(0)
> -G''(s) \quad \textrm{for all }s\neq 0.
$$
The results now follow from Corollary~\ref{Cor:unique_crt}.
\end{proof}

Let us give some examples of weights $A$ and $b$ for which $D^2\mathcal{E}(0)\geq 0$
(\textit{i.e.,} for which $\lambda_1(A,b, \Omega) \geq - G''(0)$).
We know that in this case, if in addition $-G''(0) > -G''(s)$ for all $s\neq 0$, then $\mathcal{E}$ admits a
unique critical point in $H^1_\varphi(\Omega)$ which, furthermore, is antisymmetric
under assumption \eqref{eq:Flughafen:bis}.

\begin{Example}\label{ex6:2}
\begin{enumerate}
\item[(i)]
The condition $\lambda_1 (A,b,\Omega) \geq - G''(0)$ is always
satisfied for small domains $\Omega$. Indeed, by setting
$\varepsilon\Omega := \{ \varepsilon x \, \colon \, x \in \Omega \}$,
and using the assumption~\eqref{eq:UnifCoercivity} together
with $b \in L^{\infty} (\Omega)$,
we get $\lambda_1 (A,b, \varepsilon\Omega)
\geq c_0\|b\|_\infty^{-1} \varepsilon^{-2} \lambda_1 ({\rm Id}, 1, \Omega)$.
Hence, we have $\lambda_1 (A,b,\varepsilon\Omega) \geq - G''(0)$ for $\varepsilon$ small.

\item[(ii)]
Consider the weights
$$
  A(x) = e^{\alpha |x|^2} {\rm Id}
  \quad\textrm{and}\quad
  b(x) = e^{\alpha |x|^2}
  \quad\textrm{for all }x\in\Omega\subset\mathbb{R}^N,
$$
with
$2 \alpha N \geq -G''(0)$. Then we easily check that the function
$\psi(x)=e^{-\alpha |x|^2}> 0$ is a positive supersolution of the linearized problem
at $u\equiv0$:
$$
  - \hbox{div } \Big( A(x) \nabla \psi \Big) + b(x) G''(0) \psi \geq 0 \quad \textrm{in }\Omega.
$$
It is then standard to conclude that $D^2\mathcal{E}(0)\geq 0$ (multiply the previous inequality by
$\xi^2/\psi$ with $\xi\in H^1_0(\Omega)$, integrate by parts, and use Cauchy-Schwarz).
\item[(iii)]
Note that, for $\alpha>0$, the previous weight $e^{\alpha|x|^2}$ is log-convex.
More generally, assume now
$$
  \langle A(x) \xi, \xi \rangle
  \geq  e^{\alpha|x|^2} |\xi|^2 ,
  \quad
  b(x) \leq e^{\alpha|x|^2}  ,
  \quad
  \textrm{and}
  \quad
  2 \alpha N \geq -G''(0) .
$$
Then
$$
  \frac{\int_{\Omega} \langle A(x) \nabla \xi, \nabla \xi \rangle\,dx}{\int_{\Omega} b(x) \xi^2\,dx}
  \geq
  \frac{\int_{\Omega} e^{\alpha|x|^2} |\nabla \xi|^2\,dx}{\int_{\Omega} e^{\alpha|x|^2} \xi^2\,dx} \,.
$$
By the argument in (ii), this provides examples of weights (which are not necessarily log-convex)
for which critical points are antisymmetric on any domain $\Omega$ (even
in dimension $N=1$).
\item[(iv)]
Assume $|A|$, $b \in L^{\infty} (\mathbb R^N)$ and $\inf_{\mathbb R^N} b >0$.
Then, we easily find a constant $C>0$ such that
$\lambda_1 (A, b, \Omega) \leq C \lambda_1 (\textrm{Id},1,\Omega)$. Therefore, if $G''(0)<0$, for
large domains $\Omega$ it holds that $\lambda_1(A,b,\Omega)<-G''(0)$ (\textit{i.e.},
$D^2\mathcal{E}(0)(\xi,\xi)<0$ for some $\xi\in H^1_0(\Omega)$). Hence, for some $\varphi\in C^1(\overline{\Omega})$
there are minimizers of $\mathcal{E} (\cdot, \Omega)$ in $H^{1}_\varphi (\Omega)$
which are not antisymmetric (by Proposition~\ref{Proposition:Marc1}~(ii)).
\end{enumerate}
\end{Example}

\begin{Remark}\label{rem6:7}
The uniform coercivity condition~\eqref{eq:UnifCoercivity}  is used to
guarantee the  existence of minimizers in Proposition~\ref{prop:ExistenceMin}.
It can be relaxed, and one can consider weights $A(x)$ that
either vanish at some point, or that are not uniformly coercive.
Let us briefly make two observations in this direction.

Let $\Omega \subset\mathbb{R}^N$ a bounded smooth domain that is star-shaped with
respect to the origin. Assume $N\geq 2$ and that there exist $\alpha>0$ and
$\beta\in\mathbb{R}$ such that
$$
   A(\tau x) = \tau^{\alpha} A(x),
   \qquad
   b(\tau x) = \tau^{\beta} b(x)
   \quad
   \textrm{for all } \tau > 0\textrm{ and }x\in\Omega,
$$
with $A$ and $b$ continuous in $\overline{\Omega}$ and positive in $\Omega$.
Then a simple scaling argument shows that
$$
  \lambda_1(A, b, \tau \Omega ) = \tau^{\alpha - \beta - 2} \lambda_1 (A,b, \Omega).
$$

By the weighted Hardy inequality in $H^1(\Omega;|x|^\alpha\,dx)$ we have
\begin{equation}\label{lambda_positive}
\lambda_1(A, b, \Omega ) > 0\quad\textrm{when }\alpha-\beta\leq 2.
\end{equation}
This can be proved integrating with spherical coordinates and using, on every ray,
the typical argument that gives the classical Hardy inequality (through integration
by parts and Cauchy-Schwarz inequality).

Now, if $G''(0)\geq 0$ then $D^2\mathcal{E}(0)\geq 0$ independently of the domain.
In addition, if $G''(0)<0$ and $\alpha - \beta < 2$ then we see by scaling
that $\lambda_1(A, b, \Omega )$ is large for small domains $\Omega$. In particular,
$D^2\mathcal{E}(0) \geq 0$ in this case. On the contrary, when $\alpha - \beta =2$,
$\lambda_1(A,b, \Omega)$ is invariant by dilations of the domain.

An important example of the previous situation is
the weight $A(s,t) = (st)^{m-1}{\rm Id}$, $b(s,t) =  (st)^{m-1}$ in
$\Omega = (0,L)^{2}\subset\mathbb{R}^{2}$
related to the De Giorgi conjecture, as discussed in the Introduction.
In this case $\alpha-\beta=0<2$, and thus we have uniqueness and antisymmetry in
small domains.

Another relevant case of homogeneous weights is given by radial weights.
If  $A(x) = |x|^{\alpha} {\rm Id}$ and $b(x) = |x|^{\beta}$, with $\alpha>2-N$ and $\beta>-N$,
then the value of $\lambda_1(A,b, \Omega)$ is known by the weighted Hardy
inequality in $H^1(\Omega;|x|^\alpha\,dx)$ in the critical case
$\alpha - \beta = 2$ (recall that $0\in\Omega$ since $\Omega$ is assumed to be star-shaped):
$$
  \lambda_1 (A,b,\Omega) = \frac{(N-2+\alpha)^2}{4}.
$$
Hence for this kind of weights, if
\begin{equation}\label{cond_uniq}
   \frac{(N-2+\alpha)^2}{4} \geq - G''(0) > - G''(s)\quad \textrm{for all }s\neq 0
\end{equation}
holds, then $D^2\mathcal{E}(0)\geq0$ and $\mathcal{E}(\cdot,\Omega)$ admits a unique
critical point in $H^1(\Omega;|x|^\alpha\,dx)$. The proof is the same as that of
Proposition~\ref{Proposition:Marc1} with $H^1_\varphi(\Omega)$ replaced by the previous
weighted Sobolev space.

Note that \eqref{cond_uniq} is independent of the domain $\Omega$, and that the
first inequality is satisfied for large dimensions $N$ for any $\alpha>0$.

\end{Remark}

Let us conclude with some remarks in dimension $N=1$. In this case,
we can characterize the weights $a,b$ defined on $\mathbb R$ for which
$\lambda_1 (a,b,I) \geq C$ for some positive constant $C$
independent of the interval $I$. This characterization is similar to
the Muckenhoupt's condition available for Hardy's type inequalities with
weights (see~\cite{Muckenhoupt}). More specifically, referring to Opic
and Kufner~\cite[p.~93]{OpicKufner}, given even weights $a,b$, define
on each interval $I=(-L,L)$ the constant
\begin{equation}\label{M(a,b,I)}
   M (a,b,I) :=
   \sup_{\alpha, \beta \in I} \left\{
            \left( \int_{\alpha}^{\beta} b \right) \left( \int_{\max\{ |\alpha|, |\beta| \} }^{L} \! \!a^{-1} \right)
   \right\} \,.
\end{equation}
Then,
\begin{equation}\label{opic}
    \frac{1}{16 M (a,b,I)} \, \leq \,
    \lambda_1 (a,b,I) \, \leq \,
    \frac{4}{M (a,b,I)}.
\end{equation}
In particular
$$
  \lim_{L \to \infty} \lambda_1(a,b, (-L,L))=0
  \quad \textrm{if and only if} \quad\lim_{L \to \infty} M(a,b,(-L,L)) = \infty \,.
$$

Note that the quantity appearing in \eqref{M(a,b,I)} already appeared in \eqref{eq:KindMuckenhoupt}.

If $a^{-1} , b \in L^1 (\mathbb R)$ we immediately deduce, from
Corollary~\ref{Cor:unique_crt},  \eqref{M(a,b,I)}, and \eqref{opic}, the following result.
\begin{Proposition}
Assume $a^{-1} , b \in L^1 (\mathbb R)$ and that \eqref{eq:A} and \eqref{newG:bis} hold. If
$$
   \frac{1}{ 16 \| a^{-1} \|_{L^1 (\mathbb R)} \| b \|_{L^1 (\mathbb R)} }
   \geq - G'' (0) >  - G''(s) \quad \textrm{for all }s\neq 0,
$$
then the functional $\mathcal{E} (\cdot, (-L,L))$ admits a unique critical point in $H^1_m((-L,L))$
for any $m >0$.
\end{Proposition}


\end{document}